\documentclass[a4paper,12pt]{article}
\usepackage[T1]{fontenc}
\usepackage[utf8]{inputenc}
\usepackage[english]{babel}
\usepackage{amsmath,amssymb,amsfonts,amsthm}
\usepackage{lmodern}
\usepackage{fullpage}

\usepackage[size=footnotesize]{caption}
\usepackage{subcaption}

\usepackage{cite}
\usepackage{enumerate}
\usepackage{algorithmic}
\usepackage{algorithm}

\usepackage[hidelinks]{hyperref}
\usepackage{graphicx}
\usepackage[nameinlink]{cleveref}

\newtheorem{theorem}{Theorem}
\theoremstyle{definition}
\newtheorem{lemma}{Lemma}
\newtheorem{remark}{Remark}

\newcommand{\n}[1]{\| #1 \|} 
\renewcommand{\a}{\alpha}
\renewcommand{\b}{\beta}
\renewcommand{\c}{\gamma}
\renewcommand{\d}{\delta}
\newcommand{\D}{\Delta}
\newcommand{\la}{\lambda}
\renewcommand{\t}{\tau}
\newcommand{\s}{\sigma}

\newcommand{\e}{\varepsilon}
\renewcommand{\th}{\theta}

\newcommand{\x}{\bar x}

\newcommand{\R}{\mathbb R}

\newcommand{\N}{\mathbb N}
\newcommand{\Z}{\mathbb Z}
\newcommand{\g}{\geq}
\renewcommand{\l}{\leq}

\newcommand{\lr}[1]{\langle #1\rangle} 

\newcommand{\hx}{\hat x}
\newcommand{\hy}{\hat y}

\DeclareMathOperator{\prox}{prox}
\DeclareMathOperator{\argmin}{argmin}
\DeclareMathOperator{\dom}{dom}

\begin{document}
\title{A first-order primal-dual algorithm with linesearch}

\author{Yura Malitsky$^*$ \and Thomas Pock\thanks{Institute for
Computer Graphics and Vision, Graz University of Technology, 8010
Graz,Austria. E-mail:~\href{mailto:y.malitsky@gmail.com}{y.malitsky@gmail.com},
\href{mailto:pock@icg.tugraz.at}{pock@icg.tugraz.at}}}
\date{}

\maketitle
\begin{abstract}
    The paper proposes a linesearch for a primal-dual method. Each
iteration of the linesearch requires to update only the dual (or primal)
variable. For many problems, in particular for regularized least
squares, the linesearch does not require any additional matrix-vector
multiplications. We prove convergence of the proposed method under
standard assumptions. We also show an ergodic $O(1/N)$ rate of
convergence for our method.  In case one or both of the prox-functions
are strongly convex, we modify our basic method to get a  better convergence rate.
Finally, we propose a linesearch for a saddle point
problem with an additional smooth term. Several numerical experiments confirm the
efficiency of our proposed methods.
\end{abstract}

{\small

\smallskip

 }

 \bigskip

In this work we propose a linesearch procedure for the primal-dual
algorithm (PDA) that was introduced in \cite{chambolle2011first}. It
is a simple first-order method that is widely used for solving
nonsmooth composite minimization problems. Recently, it was shown the
connection of PDA with proximal point algorithms \cite{he-yuan:2012}
and ADMM \cite{chambolle2016introduction}. Some generalizations of the
method were considered in
\cite{pock:ergodic,goldstein2013adaptive,Condat2013}. A survey of
possible applications of the algorithm can be found in
\cite{komodakis:playing,chambolle2016introduction}.

The basic form of PDA uses fixed step sizes during all
iterations. This has several drawbacks.
First, we have to compute the norm of the operator, which may be quite
expensive for large scale dense matrices.  Second, even if we know
this norm, one can often use larger steps, which usually yields a
faster convergence. As a remedy for the first issue one can use a
diagonal precondition \cite{pock2011diagonal}, but still there is no
strong evidence that such a precondition improves or at least does not
worsen the speed of convergence of PDA.  Regarding the second issue,
as we will see in our experiments, the speed improvement gained by
using the linesearch sometimes can be significant.

Our proposed analysis of PDA exploits the idea of recent works
\cite{malitsky15,malitsky2016proximal} where several algorithms are
proposed for solving a monotone variational inequality. Those
algorithms are different from the PDA; however, they also use a
similar extrapolation step. Although our analysis of the primal-dual
method is not so elegant as, for example, that in \cite{he-yuan:2012},
it gives a simple and a cheap way to incorporate the linesearch for
defining the step sizes.  Each inner iteration of the linesearch
requires updating only the dual variables. Moreover, the step sizes
may increase from iteration to iteration. We prove the convergence of
the algorithm under quite general assumptions. Also we show that in
many important cases the PDAL (primal-dual algorithm with linesearch)
preserves the complexity of the iteration of PDA. In particular, our
method, applied to the regularized least-squares problems, uses the
same number of matrix-vector multiplication per iteration as the
forward-backward method or FISTA \cite{fista} (both with fixed step size) does, but
does not require to know the matrix norm and, in addition, uses
adaptive steps.

For the case when the primal or dual objectives are strongly convex, we
modify our linesearch procedure in order to construct accelerated
versions of PDA. This is done in a similar way as in
\cite{chambolle2011first}. The obtained algorithms share the same
complexity per iteration as PDAL does, but in many cases substantially outperform
PDA and PDAL.

We also consider a more general primal-dual problem which involves an
additional smooth function with Lipschitz-continuous gradient (see
\cite{Condat2013}). For this case we generalize our linesearch to
avoid knowing that Lipschitz constant.

The authors in~\cite{goldstein2013adaptive,goldstein2015adaptive} also
proposed a linesearch for the primal-dual method, with the goal to
vary the ratio between primal and dual steps such that primal and dual
residuals remain roughly of the same size. The same idea was used
in~\cite{he2000alternating} for the ADMM method.  However, we should
highlight that this principle is just a heuristic, as it is not clear
that it in fact improves the speed of convergence.  The linesearch
proposed in~\cite{goldstein2013adaptive,goldstein2015adaptive}
requires an update of both primal and dual variables, which may make
the algorithm much more expensive than the basic PDA.  Also the
authors proved convergence of the iterates only in the case when one of the
sequences $(x^k)$ or $(y^k)$ is bounded. Although this is often the
case, there are many problems which can not be encompassed by this
assumption. Finally, it is not clear how to derive accelerated
versions of that algorithm.

As a byproduct of our analysis, we show how one can use a variable ratio
between primal and dual steps and under which circumstances we can
guarantee convergence. However, it was not the purpose of this paper
to develop new strategies for varying such a ratio during iterations.

The paper is organized as follows. In the next section we introduce
the notations and recall some useful facts.
\Cref{sec:linesearch} presents our basic primal-dual algorithm
with linesearch. We prove its convergence, establish its ergodic
convergence rate and consider some particular examples of how the
linesearch works. In \cref{sec:acceler} we propose
accelerated versions of PDAL under the assumption that the primal or dual
problem is strongly convex. \Cref{sec:general_saddle} deals
with more general saddle point problems which involve an additional
smooth function.  In \cref{sec:experiment} we illustrate the
efficiency of our methods for several typical problems.

\section{Preliminaries}\label{prelim}
Let $X$, $Y$ be two finite-dimensional real vector spaces equip\-ped with an inner
product $\lr{\cdot,\cdot}$ and a norm $\n{\cdot} =
\sqrt{\lr{\cdot,\cdot}}$. We are focusing on the following problem:
\begin{equation}
    \label{saddle}
    \min_{x\in X}\max_{y \in Y}\lr{Kx,y} + g(x) - f^*(y),
\end{equation}
where
\begin{itemize}
    \item $K\colon X\to Y$ is a bounded linear operator, with the operator
    norm $L=\n{K}$;
    \item $g\colon X\to (-\infty,+\infty]$ and $f^*\colon Y\to
    (-\infty,+\infty]$  are  proper lower semicontinuous convex
    (l.s.c.)  functions;

    \item problem~\eqref{saddle} has a saddle point.
\end{itemize}
Note that $f^*$ denotes the Legendre--Fenchel conjugate of a convex
l.s.c.\ function $f$.
By a slight (but common) abuse of notation, we write
$K^*$ to denote the adjoint of the operator $K$.

Under the above assumptions, \eqref{saddle} is equivalent to the primal problem
\begin{equation}
    \label{eq:primal}
    \min_{x\in X} f(Kx)+g(x)
\end{equation}
and the dual problem:
\begin{equation}
    \label{eq:dual}
    \min_{y\in Y} f^*(y)+g^*(-K^*x).
\end{equation}

Recall that for a proper  l.s.c.\ convex function $h\colon X \to (-\infty,
+\infty]$  the proximal operator $\prox_{h}$ is defined as
$$\prox_h\colon X\to X\colon x\mapsto \argmin_z \bigl\{h(z) + \frac 1 2 \n{z-x}^2\bigr\}.$$
The following important characteristic property of the proximal
operator is well known:
\begin{equation}\label{prox_charact}
\x = \prox_{h}x \quad \Leftrightarrow \quad  \lr{\x-x, y - \x}\g h(\x) - h(y) \quad \forall y\in X.
\end{equation}
We will often use the following identity (cosine rule):
\begin{equation}
    \label{eq:cosine}
    2\lr{a-b, c-a} = \n{b-c}^2 - \n{a-b}^2 - \n{a-c}^2 \quad \forall a,b,c \in X.
\end{equation}

Let $(\hx, \hy)$ be a saddle point of problem \eqref{saddle}. Then
by the definition of the saddle point we have
\begin{align}
      & P_{\hx,\hy}(x):= g(x) - g(\hx) + \lr{K^*\hy,x-\hx} \geq
        0\quad \forall x\in X, \label{primal_gap}\\
      &  D_{\hx,\hy}(y):= f^*(y)-f^*(\hy) - \lr{K\hx,y - \hy} \geq 0
        \quad \forall y\in Y. \label{dual_gap}
    \end{align}
    The expression
    $\mathcal{G}_{\hx,\hy}(x,y) = P_{\hx,\hy}(x)+D_{\hx,\hy}(y)$ is
    known as a primal-dual gap.
In certain cases when it is clear which saddle
point is considered, we will omit the subscript in $P$, $D$, and $\mathcal{G}$.
It is also important to highlight that for fixed $(\hx,\hy)$ functions
$P(\cdot)$, $D(\cdot)$, and $\mathcal{G}(\cdot, \cdot)$ are convex.

Consider the original primal-dual method:
\begin{align*}
  y^{k+1} & = \prox_{\sigma f^*} (y^k + \sigma K \x^k )\\
  x^{k+1} & = \prox_{\t g}(x^k - \tau K^* y^{k+1})\\
  \x^{k+1} &  = x^{k+1} + \th(x^{k+1} - x^{k}).
\end{align*}
In \cite{chambolle2011first} its convergence was proved under
assumptions $\th = 1$,  $\t, \s>0$, and $\t\s L^2<1$. In the next
section we will show how to incorporate a linesearch into this method.

\section{Linesearch}\label{sec:linesearch}
The primal-dual algorithm with linesearch (PDAL) is summarized in~\Cref{alg:A1}.

\begin{algorithm}[!ht]\caption{\textit{Primal-dual algorithm with linesearch}}
    \label{alg:A1}
\begin{algorithmic}
   \STATE {\bfseries Initialization:}
Choose $x^0\in X$, $y^1\in Y$, $\t_0 > 0$,  $\mu\in (0,1), \d
\in (0,1)$, and $\b>0$. Set $\th_0 = 1$.
\STATE {\bfseries Main iteration:}
\STATE 1. Compute  $$x^{k} = \prox_{\t_{k-1} g}(x^{k-1} -
\t_{k-1} K^*y^{k}).$$
\STATE 2. Choose any $\t_k \in [\t_{k-1}, \t_{k-1}\sqrt{1+\th_{k-1}}]$ and run\\
\STATE ~~~ \textbf{Linesearch:}
\STATE ~~~ 2.a. Compute
\vspace{-3ex}
\begin{align*}
  \th_k & = \frac{\t_k}{\t_{k-1}}\\
  \x^k & = x^k + \th_k (x^k - x^{k-1})\\
  y^{k+1} & = \prox_{\b \t_k f^*}(y^k + \b\t_k K\x^k)
\end{align*}
	\vspace{-3ex}
    \STATE   ~~~2.b. Break linesearch if
\begin{equation}\label{stop_crit}
    \sqrt{\b}\t_k \n{K^*y^{k+1} -K^*y^k} \l \d \n{y^{k+1} -y^k}
\end{equation}
\vspace{-3ex}
\STATE ~~~~~ \quad Otherwise, set $\t_k:= \t_k \mu$ and go to 2.a.
\STATE ~~~\textbf{End of linesearch}

\end{algorithmic}
\end{algorithm}
Given all information from the current iterate: $x^k$, $y^k$,
$\t_{k-1}$, and $\th_{k-1}$, we first choose some trial step
$\t_k \in [\t_{k-1}, \t_{k-1}\sqrt{1+\th_{k-1}}]$, then during every
iteration of the linesearch it is decreased by $\mu\in (0,1)$. At the
end of the linesearch we obtain a new iterate $y^{k+1}$ and a step
size $\t_k$ which will be used to compute the next iterate
$x^{k+1}$. In \Cref{alg:A1} there are two opposite options: always
start the linesearch from the largest possible step
$\t_k = \t_{k-1}\sqrt{1+\th_{k-1}}$, or, the contrary, never increase
$\t_k$. Step~2 also allows us to choose compromise between them.

Note that in PDAL parameter $\b$ plays the role of the ratio
$\frac{\sigma}{\tau}$ between fixed steps in PDA. Thus, we can
rewrite the stopping criteria~\eqref{stop_crit} as
$$    \t_k \s_k \n{K^*y^{k+1} -K^*y^k}^2 \l \d^2 \n{y^{k+1} -y^k}^2,$$
where $\s_k = \b \t_k$. Of course in PDAL we can always choose fixed
steps $\t_k$, $\s_k$ with $\t_k \s_k \leq \frac{\d^2}{L^2}$ and set
$\th_k=1$. In this case PDAL will coincide with PDA, though our
proposed analysis seems to be new.  Parameter $\d$ is mainly used for
theoretical purposes: in order to be able to control
$\n{y^{k+1}-y^k}^2$, we need $\d<1$. For the experiments $\d$ should
be chosen very close to $1$.

\begin{remark}\label{rem:ls1}
    It is clear that each iteration of the linesearch requires
computation of $\prox_{\b \t_k f^*}(\cdot)$ and $K^*y^{k+1}$. As in
problem~\eqref{saddle} we can always exchange primal and dual
variables, hence it makes sense to choose for the dual variable in PDAL the one
for which the respective prox-operator is simpler to compute.  Note that during the
linesearch we need to compute  $Kx^k$ only once and then use that
$K\x^k = (1+\th_k)Kx^k - \th_k Kx^{k-1}$.
\end{remark}

\begin{remark}\label{rem:linear} Note that when $\prox_{\s
    f^*}$ is a linear (or affine) operator, the linesearch becomes extremely simple:
    it does not require any additional matrix-vector multiplications. We itemize some examples below:
\end{remark}
\begin{enumerate}
    \item $f^*(y) = \lr{c,y}$. Then it is easy to verify that $\prox_{\sigma f^*} u = u - \sigma
    c$. Thus, we have
    \begin{align*}
      y^{k+1} & = \prox_{\sigma_k f^*}(y^k+\sigma_k K\x^k) =  y^k+\sigma_k
                K\x^k-\sigma_k c \\
      K^*y^{k+1} & = K^*y^k +\sigma_k (K^*K\x^k-K^*c)
    \end{align*}

\item $f^*(y) = \frac{1}{2} \n{y - b}^2$. Then $\prox_{\sigma f^*}u =
\frac{u+\sigma b}{1+\sigma}$ and we obtain
    \begin{align*}
      y^{k+1} & = \prox_{\sigma_k f^*}(y^k+\sigma_k K\x^k) =
                \frac{y^k+\sigma_k(K\x^k +b)}{1+\sigma_k } \\
      K^*y^{k+1} & = \frac{1}{1+\sigma_k }(K^*y^k +\sigma_k(K^*K\x^k + K^*b))
    \end{align*}

    \item $f^*(y) = \delta_H(y)$, the indicator function of the
    hyperplane $H=\{u\colon \lr{u,a} = b \}$. Then $\prox_{\sigma f^*}u
    = P_H u = u + \frac{b-\lr{u,a}}{\n{a}^2}a$. And hence,
    \begin{align*}
      y^{k+1} & = \prox_{\sigma_k f^*}(y^k+\sigma_k K\x^k) =
                y^k+\sigma_kK\x^k + \frac{b-\lr{a,y^k+\sigma_k K\x^k}}{\n{a}^2}a \\
      K^*y^{k+1} & = K^*y^k +\sigma_kK^*K\x^k + \frac{b-\lr{a,y^k+\sigma_k K\x^k}}{\n{a}^2}K^*a.
    \end{align*}

\end{enumerate}
Evidently, each iteration of PDAL for all the cases above requires
only two matrix-vector multiplications. Indeed, before the linesearch
starts we have to compute $Kx^k$ and $K^*Kx^k$ and then during the
linesearch we should use the following relations:
\begin{align*}
  K\x^k & = (1+\th_k)Kx^k - \th_k Kx^{k-1}\\
  K^*K\x^k & = (1+\th_k)K^*Kx^k - \th_k K^*Kx^{k-1},
\end{align*}
where $Kx^{k-1}$ and $K^*Kx^{k-1}$ should be reused from the previous
iteration. All other operations are comparably cheap and, hence, the
cost per iteration of PDA and PDAL are almost the same. However, this
might not be true if the matrix $K$ is very sparse. For example, for
a finite differences operator the cost of a matrix-vector
multiplication is the same as the cost of a few vector-vector
additions.

One simple implication of these facts is that for the regularized
least-squares problem
\begin{equation}\label{least_square}
    \min_x \frac 1 2 \n{Ax-b}^2 + g(x),
\end{equation}
our method does not require one to know $\n{A}$ but at the same time does
not require any additional matrix-vector multiplication as it is the
case for standard first order methods with backtracking (e.g. proximal
gradient method, FISTA). In fact, we can rewrite \eqref{least_square} as
\begin{equation}\label{least_square_saddle}
    \min_x\max_y g(x) +\lr{Ax,y} - \frac{1}{2}\n{y+b}^2.
\end{equation}
Here $f^*(y) = \frac 1 2 \n{y+b}^2$ and hence, we are in the situation where
the operator $\prox_{f^*}$ is affine.

By  construction of the algorithm we simply have the following:
\begin{lemma}\label{well-defined} $ $

    \begin{description}
        \item[(i)]     The linesearch in PDAL always terminates.
        \item[(ii)]     There exists $\tau>0$ such that $\t_k> \tau$
        for all $k\geq 0$.
        \item[(iii)]     There exists $\th>0$ such that $\th_k \leq \th$
        for all $k\geq 0$.
    \end{description}

\end{lemma}
\begin{proof}
    (i) In each iteration of the linesearch $\t_k$ is multiplied by
    factor $\mu<1$. Since, \eqref{stop_crit} is satisfied for any
    $\t_k \leq\frac{\d}{\sqrt{\b}L}$, the inner loop can not run
    indefinitely.

    (ii) Without
    loss of generality, assume that $\t_0>\frac{\d \mu}{\sqrt \b L}$. Our goal is to show that
    from $\t_{k-1}>\frac{\d \mu}{\sqrt \b L}$ follows $\t_k > \frac{\d
    \mu}{\sqrt \b L}$.  Suppose that $\t_k = \t_{k-1}\sqrt{1+\th_{k-1}}\mu^i$
    for some $i\in \Z^+$. If $i=0$ then $\t_k>\t_{k-1}>\frac{\d
    \mu}{\sqrt \b L}$. If $i>0$
    then $\t_{k}'=\t_{k-1}\sqrt{1+\th_{k-1}}\mu^{i-1}$ does not
    satisfy \eqref{stop_crit}. Thus, $\t_k'>\frac{\d}{\sqrt \b L}$ and
    hence, $\t_k>\frac{\d \mu}{\sqrt \b L}$.

    (iii) From $\t_k\leq \t_{k-1}\sqrt{1+\th_{k-1}}$ and $\th_k =
    \frac{\t_k}{\t_{k-1}}$ it follows that
    $\th_k \leq \sqrt{1+\th_{k-1}}$. From this it can be easily
    concluded that $\th_k \leq \frac{\sqrt{5}+1}{2}$ for all $k\in
    \Z^+$. In fact, assume the contrary, and let $r$ be the smallest
    number such that $\th_{r}> \frac{\sqrt 5 +1}{2}$. Since $\th_0 =
    1$, we have $r\geq 1$, and hence $\th_{r-1}\geq \th_r^2 - 1 >
    \frac{\sqrt 5 +1}{2}$. This yields a contradiction.
\end{proof}

One might be interested in how many iterations the linesearch needs in
order to terminate. Of course, we can give only a priori upper
bounds. Assume that $(\t_k)$ is bounded from above by
$\t_{max}$. Consider in step~2 of \Cref{alg:A1} two opposite ways of choosing step size
$\t_k$: maximally increase it, that is, always set $\t_k = \t_{k-1}\sqrt{1+\th_{k-1}}$ or never increase it and set $\t_k=\t_{k-1}$. If the former case holds, then
after $i$ iterations of the linesearch we have $\t_k \leq
\t_{k-1}\sqrt{1+\th_{k-1}} \mu^{i-1}\leq \t_{max} \frac{\sqrt 5 +1}{2}
\mu^{i-1}$, where we have used the bound from \Cref{well-defined} (iii). Hence, if $i-1\geq \log_{\mu}\frac{2\delta }{\b L(1+\sqrt
5)\t_{max}}$, then $\t_k\leq \frac{\d}{\b L}$ and the linesearch
procedure must terminate. Similarly, if the latter case holds, then
after at most $1+ \log_{\mu}\frac{\delta }{\b L\t_{max}}$ iterations the
linesearch stops. However, because in this case $(\t_k)$ is
nonincreasing, this number is the upper bound for the total
number of the linesearch iterations.

The following is our main convergence result.
\begin{theorem} \label{th:main} Let $(x^k, y^k)$ be a sequence
    generated by PDAL. Then it is a bounded sequence in $X\times Y$ and all its
    cluster points are solutions of \eqref{saddle}.  Moreover, if
    $g|_{\dom g}$ is continuous and $(\t_k)$ is bounded from above
    then the whole sequence $(x^k,y^k)$ converges to a solution of
    \eqref{saddle}.
\end{theorem}
The condition of $g|_{\dom g}$ to be continuous is not restrictive: it
holds for any $g$ with open $\dom g$ (this
includes all finite-valued functions) or for an indicator $\d_C$ of any
closed convex set $C$. Also it holds for any separable convex l.s.c. function
(Corollary 9.15, \cite{baucomb}). The boundedness of $(\t_k)$ from above is rather
theoretical: clearly we can easily bound it in PDAL.

\begin{proof}
    Let $(\hx,\hy)$ be any saddle point of \eqref{saddle}.
    From \eqref{prox_charact} it follows that
    \begin{align}
   & \lr{x^{k+1} - x^k + \t_k K^*y^{k+1}, \hx-x^{k+1}}  \g \t_k (g(x^{k+1})-g(\hx)) \label{xn+2}\\
   & \bigl\langle \frac{1}{\b} (y^{k+1} - y^k) - \t_k K\x^k, \hy-y^{k+1}\bigr\rangle \g \t_k (f^*(y^{k+1})-f^*(\hy)). \label{yn+2}
\end{align}
Since $x^k = \prox_{\t_{k-1} g}(x^{k-1}-\t_{k-1} K^*y^k)$, we have again by \eqref{prox_charact}
that for all $x \in X$
$$\lr{x^{k} - x^{k-1} + \t_{k-1} K^*y^k , x-x^{k}} \g \t_{k-1} (g(x^{k})-g(x)).$$
After substitution in the last inequality $x = x^{k+1}$ and $x = x^{k-1}$ we get
\begin{align}
    &\lr{x^{k} - x^{k-1} + \t_{k-1} K^*y^k , x^{k+1}-x^{k}} \g \t_{k-1} (g(x^{k})-g(x^{k+1})), \label{xn_12}   \\
    &\lr{x^{k} - x^{k-1} + \t_{k-1} K^*y^k , x^{k-1}-x^{k}} \g \t_{k-1} (g(x^{k})-g(x^{k-1})).\label{xn_22}
\end{align}
Adding \eqref{xn_12}, multiplied by $\th_k = \frac{\t_k}{\t_{k-1}}$, and \eqref{xn_22}, multiplied by $\th_k^2$, we obtain
\begin{equation}\label{xn_32}
\lr{\x^{k} - x^{k} + \t_k K^*y^k, x^{k+1}-\x^{k}} \g \t_k ((1+\th_k)g(x^{k})-g(x^{k+1})-\th_k g(x^{k-1})),
\end{equation}
where we have used that $\x^k = x^k + \th_k (x^k - x^{k-1})$.

Consider the following identity:
\begin{equation}
    \label{ident2}
    \t_k\lr{K^*y^{k+1} - K^*\hy, \x^k - \hx}  - \t_k \lr{K\x^k- K\hx, y^{k+1} - \hy} = 0.
\end{equation}
Summing \eqref{xn+2}, \eqref{yn+2}, \eqref{xn_32}, and \eqref{ident2}, we get
\begin{multline}\label{ls:adding}
    \lr{x^{k+1} - x^k, \hx-x^{k+1}} +  \frac 1 \b \lr{y^{k+1} - y^k
    , \hy-y^{k+1}} +  \lr{\x^{k} - x^{k}, x^{k+1}-\x^{k}}  \\ +
    \t_k \lr{K^*y^{k+1}- K^* y^k, \x^k - x^{k+1}} - \t_k \lr{K^*y,\x^k - \hx} + \t_k \lr{K\hx, y^{k+1} -
    y} \\ \g \t_k \bigl(f^*(y^{k+1}) - f^*(\hy)  + (1+\th_k)g(x^k) -\th_k g(x^{k-1}) -g(\hx)\bigr).
\end{multline}
Using that
\begin{equation}\label{ls:dual}
    f^*(y^{k+1}) - f^*(\hy) - \lr{K\hx,y^{k+1} - \hy} = D_{\hx,\hy}(y^{k+1})
\end{equation}
and
\begin{align}\label{ls:primal}
    (1 + & \th_k)g(x^k) - \th_k g(x^{k-1}) -g(\hx) +
           \lr{K^*y,\x^k-\hx}\nonumber \\
    & = (1+\th_k)\left(g(x^k)-g(\hx) + \lr{K^*\hy, x^k-\hx}\right)
    -\th_k\left(g(x^{k-1})-g(\hx) + \lr{K^*\hy, x^{k-1}-\hx}\right)
      \nonumber \\
    & = (1+\th_k)P_{\hx,\hy}(x^{k}) - \th_k P_{\hx,\hy}(x^{k-1}),
\end{align}
we can rewrite \eqref{ls:adding} as (we will not henceforth write
the subscripts for $P$ and $D$ unless it is unclear)
\begin{multline}\label{ls:adding2}
     \lr{x^{k+1} - x^k, \hx-x^{k+1}} +  \frac 1 \b \lr{y^{k+1} - y^k
    , \hy-y^{k+1}} + \lr{\x^{k} -
    x^{k}, x^{k+1}-\x^{k}}  \\ +
    \t_k \lr{K^*y^{k+1}- K^* y^k, \x^k - x^{k+1}}   \g \t_k
    ((1+\th_k)P(x^{k}) - \th_k P(x^{k-1}) + D(y^{k+1})).
\end{multline}
Let $\varepsilon_k$ denotes the right-hand side of \eqref{ls:adding2}.
Using the cosine rule for every item in the first line of
\eqref{ls:adding2}, we obtain
\begin{multline}\label{ineq:42}
   \frac{1}{2}(\n{x^k-\hx}^2 - \n{x^{k+1}-\hx}^2 - \n{x^{k+1}-x^k}^2)
   \\ + \frac{1}{2\b} (\n{y^k-\hy}^2 -
    \n{y^{k+1}-\hy}^2 - \n{y^{k+1}-y^k}^2 )\\ + \frac 1 2 ( \n{x^{k+1}-x^k}^2 - \n{\x^k-x^k}^2
    -\n{x^{k+1}-\x^k}^2) \\ +     \t_k \lr{K^*y^{k+1}- K^* y^k, \x^k
    - x^{k+1}} \geq \e_k.
\end{multline}
By \eqref{stop_crit}, Cauchy--Schwarz, and Cauchy's inequalities we get
\begin{align*}
\t_k \lr{K^*y^{k+1}- K^* y^k, \x^k - x^{k+1}} & \leq \frac{\d}{\sqrt{\b}}
\n{x^{k+1}-\x^k}\n{y^{k+1}-y^k}  \\ & \leq \frac 1 2 \n{x^{k+1}-\x^k}^2 +
  \frac{\d^2}{2\b}\n{y^{k+1}-y^k}^2,
\end{align*}
from which we derive that
\begin{multline}\label{ls:for_summing}
   \frac 1 2 (\n{x^k-\hx}^2 - \n{x^{k+1}-\hx}^2)  + \frac{1}{2\b} (\n{y^k-\hy}^2 -
    \n{y^{k+1}-\hy}^2) \\  - \frac 1 2 \n{\x^k-x^k}^2  -
    \frac{1-\d^2}{2\b} \n{y^{k+1}-y^k}^2   \geq \e_k.
\end{multline}
Since $(\hx, \hy)$ is a saddle point,  $D(y^k)\geq 0$ and $P(x^k)\geq
0$ and  hence \eqref{ls:for_summing} yields
\begin{multline}\label{ls:no_D}
   \frac 1 2 (\n{x^k-\hx}^2 - \n{x^{k+1}-\hx}^2)  + \frac{1}{2\b} (\n{y^k-\hy}^2 -
    \n{y^{k+1}-\hy}^2) \\  - \frac 1 2 \n{\x^k-x^k}^2  -
    \frac{1-\d^2}{2\b} \n{y^{k+1}-y^k}^2   \geq \t_k((1+\th_k)P(x^k) -
    \th_k P(x^{k-1}))
\end{multline}
or, taking into account $\th_k\t_k \leq
(1+\th_{k-1})\t_{k-1}$,
\begin{multline}\label{ls:no_D2}
   \frac 1 2 \n{x^{k+1}-\hx}^2  + \frac{1}{2\b} \n{y^{k+1}-\hy}^2 + \t_k(1+\th_k)P(x^k)
   \leq \\  \frac 1 2 \n{x^k-\hx}^2 + \frac{1}{2\b} \n{y^k-\hy}^2+ \t_{k-1}(1+\th_{k-1}) P(x^{k-1}) \\- \frac 1 2 \n{\x^k-x^k}^2  -
    \frac{1-\d^2}{2\b} \n{y^{k+1}-y^k}^2
\end{multline}
From this we deduce that $(x^k)$, $(y^k)$ are bounded sequences and
$\lim_{k\to \infty}\n{\x^k-x^{k}} = 0$,
$\lim_{k\to \infty}\n{y^k-y^{k-1}} = 0$.  Also notice that $$\frac{x^{k+1}-x^k}{\t_k} =
\frac{\x^{k+1}-x^{k+1}}{\t_{k+1}}\to 0 \quad \text{as } k\to \infty,$$
where the latter holds because $\t_k$ is separated from $0$ by
\Cref{well-defined}. Let  $(x^{k_i}, y^{k_i})$ be a subsequence
that converges to some cluster point $(x^*,y^*)$. Passing to the limit
in
\begin{align*}
  \lr{\frac{1}{\t_{k_i}} (x^{k_i+1} - x^{k_i}) + K^*y^{k_i} ,
  x-x^{k_i+1}} & \geq g(x^{k_i+1})-g(x) \quad \forall x\in X,\\
   \lr{\frac{1}{\b \t_{k_i}} (y^{k_i+1} - y^{k_i}) - K\x^{k_i},
  y-y^{k_i+1}} & \geq f^*(y^{k_i+1})-f^*(y) \quad \forall y \in Y,
\end{align*}
 we obtain that $(x^*,y^*)$ is a saddle point of \eqref{saddle}.

When $g|_{\dom g}$ is continuous, $g(x^{k_i})\to g(x^*)$, and hence,
$P_{x^*, y^*}(x^{k_i})\to 0$. From \eqref{ls:no_D2} it follows that the
sequence $a_k = \frac 1 2 \n{x^{k+1}-x^*}^2  + \frac{1}{2\b}
\n{y^{k+1}-y^*}^2 + \t_k(1+\th_k)P_{x^*, y^*}(x^k)$ is monotone.
Taking into account boundedness of $(\t_k)$ and $(\th_k)$, we obtain $$\lim_{k\to \infty} a_k = \lim_{i\to \infty} a_{k_i}=0,$$
which means that $x^k\to x^*$, $y^k\to y^*$.
\end{proof}

\begin{theorem}[ergodic convergence] \label{th:ergodic}
Let  $(x^k, y^k)$ be a sequence generated by PDAL and $(\hx,\hy)$ be
any saddle point of \eqref{saddle}. Then for the ergodic
 sequence $(X^N, Y^N)$ it holds that
\begin{equation*}
\mathcal{G}_{\hx,\hy}(X^N, Y^N) \leq  \frac{1}{s_N} \left( \frac
    1 2 \n{x^1-\hx}^2  + \frac{1}{2\b} \n{y^1-\hy}^2 + \t_1 \th_1 P_{\hx,\hy}(x^0)\right),
\end{equation*}
where $s_N =   \sum_{k=1}^{N}\t_k$, $X^N = \dfrac{\t_1\th_1 x^0 + \sum_{k=1}^N\t_k \x^k}{\t_1 \th_1 + s_N}$, $Y^N = \dfrac{\sum_{k=1}^N \t_k y^k}{s_N}$.
\end{theorem}
\begin{proof}
Summing up~\eqref{ls:for_summing} from $k=1$ to $N$, we get
\begin{equation}\label{ls:after_summing}
   \frac 1 2 (\n{x^1-\hx}^2 - \n{x^{N+1}-\hx}^2)  + \frac{1}{2\b} (\n{y^1-\hy}^2 -
    \n{y^{N+1}-\hy}^2)  \geq \sum_{k=1}^N \e_k
\end{equation}
The right-hand side in \eqref{ls:after_summing} can be expressed as
\begin{align*}
\sum_{k=1}^N \varepsilon_k  =  \t_N(1+\th_N) P(x^N) & + \sum_{k=2}^{N} [(1+\th_{k-1})\t_{k-1}-\th_k
\t_k)]P(x^{k-1}) \\ &  - \th_1\t_1 P(x^0) + \sum_{k=1}^N \t_kD(y^{k+1})
\end{align*}
By convexity of $P$,
\begin{align}\label{P_convex}
 \t_N(1+\th_N) P(x^N) + & \sum_{k=2}^{N} [(1+\th_{k-1}) \t_{k-1}-\th_k
\t_k)]P(x^{k-1}) \\ & \geq  (\t_1\th_1+ s_N)
P(\frac{\t_1(1+\th_1) x^1 + \sum_{k=2}^N\t_k \x^k}{\t_1 \th_1 + s_N})
\\ & = (\t_1\th_1+ s_N)
P(\frac{\t_1\th_1 x^0 + \sum_{k=1}^N\t_k \x^k}{\t_1 \th_1 + s_N}) \geq  s_N
P(X^N),
\end{align}
where $s_N =   \sum_{k=1}^{N}\t_k$. Similarly,
\begin{equation}\label{D_convex}
\sum_{k=1}^N \t_k D(y^k) \geq s_N D(\frac{\sum_{k=1}^N\t_k y^k}{s_N}) =
s_N D(Y^N).
\end{equation}
Hence,
\begin{equation*}
    \sum_{k=1}^N\e_k \geq s_N (P(X^N)+D(Y^N)) -  \t_1 \th_1 P(x^0)
\end{equation*}
and we conclude
\begin{equation*}
    \mathcal{G}(X^N, Y^N)  = P(X^N)+D(Y^N)  \leq  \frac{1}{s_N} \left( \frac
    1 2 \n{x^1-\hx}^2  + \frac{1}{2\b} \n{y^1-\hy}^2 + \t_1 \th_1
                                                  P(x^0)\right). \qedhere
\end{equation*}

\end{proof}
\bigskip
\noindent Clearly, we have the same $O(1/N)$ rate of convergence as in
\cite{chambolle2011first,pock:ergodic,goldstein2015adaptive}, though
with the ergodic sequence $(X^N,Y^N)$ defined in a different way.

Analysing the proof of Theorems~\ref{th:main} and~\ref{th:ergodic},
the reader may find out that our proof does not rely on the
proximal interpretation of the PDA~\cite{he-yuan:2012}. An obvious
shortcoming of our approach is that  deriving new
extensions of the proposed method, like inertial or relaxed versions~\cite{pock:ergodic},
still requires some nontrivial efforts. It would be interesting to
obtain a general approach for the proposed method and its possible extensions.

It is well known that in many cases the speed of convergence
of PDA crucially depends on the ratio between primal and dual steps
$\b  = \dfrac{\sigma}{\tau}$. Motivated by this,
paper~\cite{goldstein2013adaptive} proposed an adaptive strategy for how
to choose $\b$ in every iteration. Although it is not the goal of this
work to study the strategies for defining $\b$, we show that the analysis
of PDAL allows us to incorporate such strategies in a very natural way.
\begin{theorem}
     Let $(\b_k)\subset (\b_{min},\b_{max}) $ be a monotone sequence
     with $\b_{min},\b_{max} >0$ and $(x^k,y^k)$ be a sequence generated by PDAL with variable
     $(\b_k)$.     Then the statement of \Cref{th:main} holds.

 \end{theorem}
\begin{proof}
    Let $(\b_k)$ be nondecreasing. Then using $\frac{1}{\b_k}\leq
    \frac{1}{\b_{k-1}}$, we get from \eqref{ls:no_D2} that
\begin{multline}\label{adaptive:1}
   \frac 1 2 \n{x^{k+1}-\hx}^2  + \frac{1}{2\b_k} \n{y^{k+1}-\hy}^2 + \t_k(1+\th_k)P(x^k)
   \leq \\  \frac 1 2 \n{x^k-\hx}^2 + \frac{1}{2\b_{k-1}} \n{y^k-\hy}^2+ \t_{k-1}(1+\th_{k-1}) P(x^{k-1}) \\- \frac 1 2 \n{\x^k-x^k}^2  -
   \frac{1-\d^2}{2\b_k} \n{y^{k+1}-y^k}^2  .
\end{multline}
Since $(\b_k)$ is bounded from above, the
conclusion in \Cref{th:main} simply follows.

If $(\b_k)$ is decreasing, then the above arguments should be modified.
Consider \eqref{ls:no_D}, multiplied by $\b_k$,
\begin{multline}\label{adaptive:2}
   \frac{\b_k}{2} (\n{x^k-\hx}^2 - \n{x^{k+1}-\hx}^2)  + \frac{1}{2} (\n{y^k-\hy}^2 -
    \n{y^{k+1}-\hy}^2) \\  - \frac{\b_k}{2} \n{\x^k-x^k}^2  -
    \frac{1-\d^2}{2} \n{y^{k+1}-y^k}^2   \geq \b_k\t_k((1+\th_k)P(x^k) -
    \th_k P(x^{k-1})).
\end{multline}
As $\b_k< \b_{k-1}$, we have  $\th_k\b_k\t_k
\leq(1+\th_{k-1})\b_{k-1}\t_{k-1}$, which in turn implies
\begin{multline}\label{adaptive:3}
   \frac{\b_k}{2} \n{x^{k+1}-\hx}^2  + \frac{1}{2} \n{y^{k+1}-\hy}^2 + \t_k\b_k(1+\th_k)P(x^k)
   \leq \\  \frac{\b_{k-1}}{2} \n{x^k-\hx}^2 + \frac{1}{2} \n{y^k-\hy}^2+ \t_{k-1}\b_{k-1}(1+\th_{k-1}) P(x^{k-1}) \\- \frac{\b_k}{2} \n{\x^k-x^k}^2  -
   \frac{1-\d^2}{2} \n{y^{k+1}-y^k}^2  .
\end{multline}
Due to the given properties of $(\b_k)$, the rest is trivial.
\end{proof}

It is natural to ask if it is possible to use nonmonotone $(\b_k)$.
To answer this question, we can easily use the strategies
from~\cite{he2000alternating}, where an ADMM with variable steps was proposed. One way is to use any
$\b_k$ during a finite number of iterations and then
switch to monotone $(\b_k)$.  Another way is to  relax the monotonicity
of $(\b_k)$ to the following:
\begin{align*}
\text{there exists } (\rho_k) \subset \R_+ \quad \text{such that} \colon
  \sum_{k}\rho_k & < \infty \quad \text{and}
\\
  \b_k & \leq \b_{k-1}(1+\rho_{k}) \quad \forall k\in \N \quad \text{or } \\
  \b_{k-1} & \leq \b_{k}(1+\rho_{k}) \quad \forall k\in \N.
\end{align*}
We believe that it should be quite straightforward to prove convergence of
PDAL with the latter strategy.

\section{Acceleration}\label{sec:acceler}
It has been shown \cite{chambolle2011first} that in the case when $g$ or
$f^*$ is strongly convex, one can modify the primal-dual algorithm
and derive a better convergence rate. We show that the same holds for
PDAL. The main difference of the accelerated variant APDAL from the
basic PDAL is that now we have to vary $\b$ in every iteration.

Of course due to the symmetry of the primal and dual variables in
\eqref{saddle}, we can always assume that the primal objective is
strongly convex. However, from the computational point of view it
might not be desirable to exchange $g$ and $f^*$ in the PDAL because
of \Cref{rem:linear}. Therefore, we discuss the two cases
separately. Also notice that both accelerated algorithms below
coincide with PDAL when the parameter of strong convexity $\c = 0$.

\subsection{$g$ is strongly convex}
Assume that $g$ is $\c$--strongly convex, i.e.,
$$g(x_2)-g(x_1)\geq \lr{u, x_2-x_1} + \frac \c 2 \n{x_2-x_1}^2 \quad
\forall x_1,x_2\in X, u\in \partial g(x_1).$$
Below we assume that the parameter $\c$ is known. The following
algorithm (APDAL) exploits the strong convexity of $g$:
\begin{algorithm}[!ht]\caption{\textit{Accelerated primal-dual
    algorithm with linesearch: $g$ is strongly convex}}
    \label{alg:A2}
\begin{algorithmic}
   \STATE {\bfseries Initialization:}
Choose $x^0\in X$, $y^1\in Y$, $\mu\in (0,1)$, $\t_0>0$, $\b_0>0$. Set $\th_0 = 1$.
\STATE {\bfseries Main iteration:}
\STATE 1. Compute
\begin{align*}
x^{k} & = \prox_{\t_{k-1} g}(x^{k-1} -
        \t_{k-1} K^*y^{k}) \\
       \b_k & = \b_{k-1}(1+\c \t_{k-1})
\end{align*}
\vspace{-3ex}
\STATE 2. Choose any $\t_k \in
[\t_{k-1}\sqrt{\frac{\b_{k-1}}{\b_k}},
\t_{k-1}\sqrt{\frac{\b_{k-1}}{\b_k}(1+\th_{k-1})}]$ and run

\STATE ~~~ \textbf{Linesearch:}
\STATE ~~~ 2.a. Compute
\vspace{-3ex}
\begin{align*}
  \th_k & = \frac{\t_k}{\t_{k-1}}\\
  \x^k & = x^k + \th_k (x^k - x^{k-1})\\
  y^{k+1} & = \prox_{\b_k \t_k f^*}(y^k + \b_k\t_k K\x^k)
\end{align*}
	\vspace{-3ex}
    \STATE   ~~~2.b. Break linesearch if
    \begin{equation}\label{acc:stop-crit}
        \sqrt{\b_k} \t_k \n{K^*y^{k+1} -K^*y^k} \l \n{y^{k+1} -y^k}
    \end{equation}
	\vspace{-3ex}
\STATE ~~~~~ \quad Otherwise, set $\t_k:= \t_k \mu$ and go to 2.a.
\STATE ~~~\textbf{End of linesearch}

\end{algorithmic}
\end{algorithm}

Note that in contrast to PDAL, we set $\d = 1$, as in any case we will not be able
to prove convergence of $(y^k)$.

Instead of \eqref{xn+2}, now one can use the stronger inequality
\begin{equation}\label{acc:xn}
  \lr{x^{k+1} - x^k + \t_k K^*y^{k+1}, \hx-x^{k+1}}  \g \t_k
  (g(x^{k+1})-g(\hx) + \frac \c 2 \n{x^{k+1}-\hx}^2).
\end{equation}
In turn,  \eqref{acc:xn} yields a stronger version of
\eqref{ls:for_summing} (also with $\b_k$ instead of $\b$):
\begin{multline}\label{acc:for_summing}
   \frac{1}{2} (\n{x^k-\hx}^2 - \n{x^{k+1}-\hx}^2)  + \frac{1}{2\b_k} (\n{y^k-\hy}^2 -
    \n{y^{k+1}-\hy}^2) \\  - \frac 1 2 \n{\x^k-x^k}^2    \geq \e_k + \frac{\c\t_k}{2}\n{x^{k+1}-\hx}^2
\end{multline}
or, alternatively,
\begin{multline}\label{acc:alternative}
   \frac{1}{2} \n{x^k-\hx}^2  + \frac{1}{2\b_k} \n{y^k-\hy}^2 -
   \frac 1 2 \n{\x^k-x^k}^2   \\ \geq \e_k +     \frac{1+\c\t_k}{2}\n{x^{k+1}-\hx}^2
    +  \frac{\b_{k+1}}{\b_k}  \frac{1}{2\b_{k+1}} \n{y^{k+1}-\hy}^2.
\end{multline}
Note that the algorithm provides that $\frac{\b_{k+1}}{\b_k} = 1+\c\t_k$.
For brevity let
\begin{align*}
  A_k  & =    \frac{1}{2} \n{x^k-\hx}^2  + \frac{1}{2\b_k}
         \n{y^k-\hy}^2.
\end{align*}
Then from \eqref{acc:alternative} follows
$$ \frac{\b_{k+1}}{\b_k}A_{k+1}+\e_k \leq A_k$$
or $$\b_{k+1}A_{k+1}+ \b_k\e_k \leq \b_kA_k.$$
Thus, summing the above from $k=1$ to $N$, we get
\begin{equation}
    \b_{N+1}A_{N+1}+\sum_{k=1}^N\b_k\e_k \leq \b_1 A_1.
\end{equation}
Using the convexity in the same way as in \eqref{P_convex},
\eqref{D_convex}, we obtain
\begin{multline}
\sum_{k=1}^N \b_k\e_k  =  \b_N\t_N(1+\th_N) P(x^N)+ \sum_{k=2}^{N} [(1+\th_{k-1})\b_{k-1}\t_{k-1}-\th_k
\b_k\t_k)]P(x^{k-1})  \\ - \th_1\b_1\t_1 P(x^0) + \sum_{k=1}^N \b_k\t_kD(y^{k+1})
 \geq s_N
(P(X^N) + D(Y^N)) - \th_1\s_1 P(x^0),
\end{multline}
where
\begin{align*}
  \s_k & = \b_k\t_k   && \, s_N  =  \sum_{k=1}^{N} \s_k\\
  X^N & = \frac{\s_1 \th_1 x^0 + \sum_{k=1}^N \s_k \x^k}{\s_1\th_1 +
        s_N}  && Y^N = \frac{\sum_{k=1}^N \s_k y^{k+1}}{s_N}
\end{align*}
Hence,
\begin{equation} \label{acc1_last_main}
    \b_{N+1}A_{N+1}+s_N\mathcal{G}(X^N,Y^N) \leq \b_1 A_1 +\th_1\s_1 P(x^0) .
\end{equation}
From this we deduce that the
sequence $(\n{y^{k}-\hy})$ is bounded and
\begin{align*}
  \mathcal{G}(X^N,Y^N) & \leq \frac{1}{s_{N}} (\b_1 A_1 +\th_1\s_1
  P(x^0))\\
  \n{x^{N+1}-\hx}^2  & \leq \frac{1}{\b_{N+1}} (\b_1 A_1 +\th_1\s_1 P(x^0)).
\end{align*}

Our next goal is to derive asymptotics for $\b_N$ and
$s_N$. Obviously, \eqref{acc:stop-crit} holds for any $\t_k$, $\b_k$
such that $\t_k \leq \frac{1}{\sqrt{\b_k}L}$. Since in each iteration
of the linesearch we decrease $\t_k$ by factor of $\mu$, $\t_k$ can not be
less than $\frac{\mu}{\sqrt{\b_k}L}$.
Hence, we have
\begin{equation}\label{b_k}
    \b_{k+1} = \b_k (1+\c \t_k) \geq \b_k (1+\c
    \frac{\mu}{L\sqrt{\b_k}}) = \b_k + \frac{\c\mu}{L}\sqrt{\b_k}  .
\end{equation}
By induction, one can show that there exists $C>0$ such that $\b_k
\geq Ck^2$ for all $k>0$. Then for some constant $C_1>0$ we have
$$\n{x^{N+1}-\hx}^2 \leq \frac{C_1}{(N+1)^2}.$$
From \eqref{b_k} it follows that $\b_{k+1}-\b_k \geq \frac{\c
\mu}{L}\sqrt{C} k$. As $\s_k = \frac{\b_{k+1}-\b_k}{\c} $, we
obtain $\s_k\geq \frac{\mu}{L} \sqrt C k$ and thus $s_N = \sum_{k=1}^N
\s_k = O(N^2)$. This means that for some constant $C_2>0$
$$  \mathcal{G}(X^N,Y^N)  \leq \frac{C_2}{N^2}.$$

We have shown the following result:
\begin{theorem}\label{th:str_conv1}
 Let  $(x^k, y^k)$ be a sequence generated by \Cref{alg:A2}. Then
$\n{x^N-\hx}= O(1/N)$ and $\mathcal G(X^N,Y^N) =
 O(1/N^2)$.
\end{theorem}

\subsection{$f^*$ is strongly convex}
The case when $f^*$ is $\c$--strongly convex can be treated in a
similar way.

\begin{algorithm}[!ht]\caption{\textit{Accelerated primal-dual
    algorithm with linesearch: $f^*$ is strongly convex}}
    \label{alg:A3}
\begin{algorithmic}
   \STATE {\bfseries Initialization:}
   Choose $x^0\in X$, $y^1\in Y$,  $\mu\in (0,1)$, $\t_0>0$,  $\b_0>0$. Set $\th_0 = 1$.
\STATE {\bfseries Main iteration:}
\STATE 1. Compute
\begin{align*}
x^{k} & = \prox_{\t_{k-1} g}(x^{k-1} -
        \t_{k-1} K^*y^{k}) \\
        \b_k & = \frac{\b_{k-1}}{1+ \c \b_{k-1}\t_{k-1}}
\end{align*}
\vspace{-3ex}
\STATE 2. Choose any $\t_k \in [\t_{k-1}, \t_{k-1}\sqrt{1+\th_{k-1}}]$ and run\\
\STATE ~~~ \textbf{Linesearch:}
\STATE ~~~ 2.a. Compute
\vspace{-3ex}
\begin{align*}
  \th_k & = \frac{\t_k}{\t_{k-1}},\quad \s_k = \b_k \t_k\\
  \x^k & = x^k + \th_k (x^k - x^{k-1})\\
  y^{k+1} & = \prox_{\s_k f^*}(y^k + \s_k K\x^k)
\end{align*}
	\vspace{-3ex}
    \STATE   ~~~2.b. Break linesearch if
    $$\sqrt{\b_k} \t_k \n{K^*y^{k+1} -K^*y^k} \l \n{y^{k+1} -y^k}$$
	\vspace{-3ex}
\STATE ~~~~~ \quad Otherwise, set $\t_k:= \t_k \mu$ and go to 2.a.
\STATE ~~~\textbf{End of linesearch}

\end{algorithmic}
\end{algorithm}

Note that again we set $\d = 1$.

Dividing \eqref{ls:for_summing} over $\t_k$ and taking into account
the strong convexity of $f^*$, which has to be used in~\eqref{yn+2}, we
deduce
\begin{multline}\label{ls:for_summing_accel_dual}
   \frac{1}{2\t_k} (\n{x^k-\hx}^2 - \n{x^{k+1}-\hx}^2)  + \frac{1}{2\s_k} (\n{y^k-\hy}^2 -
    \n{y^{k+1}-\hy}^2) \\  - \frac{1}{2\t_k} \n{\x^k-x^k}^2
     \geq \frac{\e_k}{\t_k} + \frac{\c}{2} \n{y^{k+1}-\hy}^2,
 \end{multline}
 which can be rewritten as
\begin{multline}\label{ls:for_summing_accel_dual_altern}
   \frac{1}{2\t_k} \n{x^k-\hx}^2   + \frac{1}{2\s_k} \n{y^k-\hy}^2  - \frac{1}{2\t_k} \n{\x^k-x^k}^2 \\
     \geq \frac{\e_k}{\t_k}  +   \frac{\t_{k+1}}{\t_k}
     \frac{1}{2\t_{k+1}} \n{x^{k+1}-\hx}^2  + \frac{\s_{k+1}}{\s_k}(1+ \c\s_k) \frac{1}{2\s_{k+1}}\n{y^{k+1}-\hy}^2,
 \end{multline}
Note that by construction of $(\b_k)$ in \Cref{alg:A3}, we have
 $$ \frac{\t_{k+1}}{\t_k} = \frac{\s_{k+1}}{\s_k}(1+ \c\s_k).$$ 
Let $A_k =    \frac{1}{2\t_k} \n{x^k-\hx}^2   + \frac{1}{2\s_k}
\n{y^k-\hy}^2 $. Then \eqref{ls:for_summing_accel_dual_altern} is
equivalent to
$$\frac{\t_{k+1}}{\t_k} A_{k+1} +\frac{\e_k}{\t_k}\leq A_k -
\frac{1}{2\t_k} \n{\x^k-x^k}^2 $$
or
$$\t_{k+1} A_{k+1} +\e_k \leq \t_k A_k -
\frac{1}{2} \n{\x^k-x^k}^2. $$
Finally, summing the above from $k=1$ to $N$, we get
\begin{equation}\label{acc2:last_main}
 \t_{N+1} A_{N+1} + \sum_{k=1}^N \e_k \leq \t_1 A_1 - \frac 1 2
  \sum_{k=1}^N\n{\x^k - x^k}^2
\end{equation}
From this we conclude that the sequence $(x^k)$ is bounded, $\lim_{k\to
\infty}\n{\x^k - x^k} = 0$, and
\begin{align*}
\mathcal{G}(X^N,Y^N)  & \leq \frac{1}{s_{N}} (\t_1 A_1 +\th_1\t_1
  P(x^0)), \\
  \n{y^{N+1}-\hy}^2  &  \leq \frac{\s_{N+1}}{\t_{N+1}}(\t_1 A_1 + \th_1\t_1
                    P(x^0)) = \b_{N+1}(\t_1 A_1 + \th_1\t_1
                    P(x^0)),
\end{align*}
where $X^N$, $Y^N$, $s_N$  are the same as in \Cref{th:ergodic}.

Let us turn to the derivation of the asymptotics of
$(\t_{N})$ and
$(s_N)$. Analogously, we have
that $\t_k \geq \frac{\mu}{\sqrt{\b_k} L}$ and thus
$$\b_{k+1} = \frac{\b_k}{1+\c \b_k \t_k}\leq
\frac{\b_k}{1+\c \frac{\mu}{L} \sqrt{\b_k}}.$$
Again it is not difficult to show by induction that $\b_k \leq \frac{C}{
k^2}$ for some constant $C>0$. In fact, as $\phi(\b) = \frac{\b}{1+\c
\frac{\mu}{L} \sqrt{\b}}$ is increasing, it is sufficient to show that
$$\b_{k+1}\leq \frac{\b_k}{1+\c \frac{\mu}{L} \sqrt{\b_k}}\leq
\frac{\frac{C}{k^2}}{1+\c \frac{\mu}{L} \sqrt{\frac{C}{k^2}} }\leq \frac{C}{(k+1)^2}. $$
The latter inequality is equivalent to
$\sqrt{C} \geq (2+\dfrac 1 k)\dfrac{L}{\c \mu}$, which obviously holds for $C$ large
enough (of course $C$ must also satisfy the induction basis).

The obtained asymptotics for $(\b_k)$ yields $$\t_k \geq
\frac{\mu}{\sqrt{\b_k} L} \geq \frac{\mu k}{\sqrt C L}, $$
from which we deduce $s_N = \sum_{k=1}^N \t_k\geq \frac{\mu}{\sqrt C
L}\sum_{k=1}^N k$. Finally, we obtain the following result:
\begin{theorem}
    Let $(x^k, y^k)$ be a sequence generated by \Cref{alg:A3}. Then $\n{y^N-\hy}= O(1/N)$ and $\mathcal G(X^N,Y^N) = O(1/N^2)$.
\end{theorem}

\begin{remark}
    In the case when both $g$ and $f^*$ are strongly convex, one can derive a
    new algorithm, combining the ideas of \Cref{alg:A2,alg:A3} (see \cite{chambolle2011first} for more details).
\end{remark}

\section{A more general problem}\label{sec:general_saddle}
In this section we show how to apply the
linesearch procedure to the more general problem
\begin{equation}
    \label{saddle_general}
    \min_{x\in X}\max_{y \in Y}\lr{Kx,y} + g(x) - f^*(y)-h(y),
\end{equation}
where in addition to the previous assumptions, we suppose that $h \colon
Y\to \R$ is a smooth convex function with $L_h$--Lipschitz--continuous gradient
$\nabla h$. Using the idea of \cite{he-yuan:2012}, Condat and V\~u in
\cite{Condat2013,vu2013splitting} proposed an extension of the primal-dual method to solve
\eqref{saddle_general}:
\begin{align*}
  y^{k+1} & = \prox_{\sigma f^*} (y^k + \sigma (K \x^k-\nabla h(y^k) ))\\
  x^{k+1} & = \prox_{\t g}(x^k - \tau K^* y^{k+1})\\
  \x^{k+1} &  = 2x^{k+1} - x^{k}.
\end{align*}
This scheme was proved to converge under the condition $\t \s
\n{K}^2\leq 1- \s L_h$.
Originally the smooth function was added to the primal part and
not to the dual as in \eqref{saddle_general}. For us it is more
convenient to consider precisely that form due to the nonsymmetry of
the proposed linesearch procedure. However, simply exchanging $\max$ and $\min$
in \eqref{saddle_general}, we recover the
form which was considered in \cite{Condat2013}.

In addition to the issues related to the operator norm of $K$ which
motivated us to derive the PDAL, here we
also have to know the Lipschitz constant $L_h$ of $\nabla
h$. This has several drawbacks. First, its computation might be
expensive. Second, our estimation of $L_h$ might be very conservative
and will result in smaller steps. Third, using local
information about $h$ instead of global $L_h$ often allows the use of larger steps.
Therefore, the introduction of a linesearch to the algorithm above is
of great practical interest.

The algorithm below exploits the same idea as  \Cref{alg:A1}
does. However, its stopping criterion is more involved. The interested
reader may identify it as a combination of the stopping criterion
\eqref{stop_crit} and the descent lemma for smooth functions $h$.

\begin{algorithm}[!ht]\caption{\textit{General primal-dual algorithm with linesearch}}
    \label{alg:A4}
\begin{algorithmic}
   \STATE {\bfseries Initialization:}
Choose $x^0\in X$, $y^1\in Y$, $\t_0 > 0$,  $\mu\in (0,1), \d
\in (0,1)$ and $\b_0>0$. Set $\th_0 = 1$.
\STATE {\bfseries Main iteration:}
\STATE 1. Compute  $$x^{k} = \prox_{\t_{k-1} g}(x^{k-1} -
\t_{k-1} K^*y^{k}).$$
\STATE 2. Choose any $\t_k\in[\t_{k-1}, \t_{k-1}\sqrt{1+\th_{k-1}}]$ and run\\
\STATE ~~~ \textbf{Linesearch:}
\STATE ~~~ 2.a. Compute
\vspace{-3ex}
\begin{align*}
  \th_k & = \frac{\t_k}{\t_{k-1}},\quad \s_k = \b \t_k\\
  \x^k & = x^k + \th_k (x^k - x^{k-1})\\
  y^{k+1} & = \prox_{\s_k f^*}(y^k + \s_k (K\x^k-\nabla h(y^k)))
\end{align*}
	\vspace{-3ex}
    \STATE   ~~~2.b. Break linesearch if
\begin{multline}\label{s5:stop_crit}
    \t_k \s_k \n{K^*y^{k+1} -K^*y^k}^2 +
    2\s_k[h(y^{k+1})-h(y^k)-\lr{\nabla h(y^k), y^{k+1}-y^k}] \\ \leq
    \d \n{y^{k+1} -y^k}^2
\end{multline}
\vspace{-3ex}
\STATE ~~~~~ \quad Otherwise, set $\t_k:= \t_k \mu$ and go to 2.a.
\STATE ~~~\textbf{End of linesearch}
\end{algorithmic}
\end{algorithm}

Note that in the case $h\equiv 0$, \Cref{alg:A4} corresponds exactly
to \Cref{alg:A1}.

We briefly sketch the proof of convergence.
        Let $(\hx,\hy)$ be any saddle point of \eqref{saddle}.
  Similarly to \eqref{xn+2}, \eqref{yn+2}, and \eqref{xn_32}, we get
    \begin{align*}
   & \lr{x^{k+1} - x^k + \t_k K^*y^{k+1}, \hx-x^{k+1}}  \g \t_k (g(x^{k+1})-g(\hx))\\
   & \lr{\frac{1}{\b} (y^{k+1} - y^k) - \t_k K\x^k+\t_k  \nabla h(y^k), \hy-y^{k+1}} \g \t_k (f^*(y^{k+1})-f^*(\hy)).\\
& \lr{\th_k(x^{k} - x^{k-1}) + \t_k K^*y^k, x^{k+1}-\x^{k}} \g \t_k
  ((1+\th_k)g(x^{k})-g(x^{k+1})-\th_k g(x^{k-1})).
    \end{align*}
Summation of the three inequalities above and identity \eqref{ident2} yields
\begin{multline}\label{s5:ls:adding}
    \lr{x^{k+1} - x^k, \hx-x^{k+1}} +  \frac 1 \b \lr{y^{k+1} - y^k
    , \hy-y^{k+1}} +  \th_k \lr{x^{k} -
    x^{k-1}, x^{k+1}-\x^{k}}  \\ +
    \t_k \lr{K^*y^{k+1}- K^* y^k, \x^k - x^{k+1}} - \t_k \lr{K^*y,\x^k - \hx} + \t_k \lr{K\hx, y^{k+1} -
    y}    \\
    + \t_k \lr{\nabla h(y^k), \hy-y^{k+1}} \g \t_k \bigl(f^*(y^{k+1}) - f^*(\hy)  + (1+\th_k)g(x^k) -\th_k g(x^{k-1}) -g(\hx)\bigr).
\end{multline}

By convexity of $h$, we have $\t_k( h(\hy) - h(y^k) - \lr{\nabla
h(y^k), \hy - y^k}) \geq 0$. Combining it with inequality
\eqref{s5:stop_crit}, divided over $2\b$, we get
$$
   \frac{\d}{2\b} \n{y^{k+1} -y^k}^2 - \frac{\t_k^2}{2}\n{K^*y^{k+1} -K^*y^k}^2 \geq
    \t_k[h(y^{k+1})-h(\hy)-\lr{\nabla h(y^k), y^{k+1}-\hy}].
    $$
Adding the above inequality to \eqref{s5:ls:adding} gives us
\begin{multline}\label{s5:adding3}
    \lr{x^{k+1} - x^k, \hx-x^{k+1}} +  \frac 1 \b \lr{y^{k+1} - y^k
    , \hy-y^{k+1}} +  \th_k \lr{x^{k} -
    x^{k-1}, x^{k+1}-\x^{k}}  \\ +
    \t_k \lr{K^*y^{k+1}- K^* y^k, \x^k - x^{k+1}} - \t_k \lr{K^*y,\x^k
    - \hx} +  \t_k \lr{K\hx, y^{k+1} -
    y}    \\
     +\frac{\d}{2\b} \n{y^{k+1}-y^k}^2 -\frac{\t_k^2}{2} \n{K^*y^{k+1} -K^*y^k}^2 \\ \geq \t_k \bigl((f^*+h)(y^{k+1}) - (f^*+h)(\hy)  + (1+\th_k)g(x^k) -\th_k g(x^{k-1}) -g(\hx)\bigr).
\end{multline}
Note that for problem \eqref{saddle_general} instead of
\eqref{dual_gap} we have to use
\begin{equation}
    \label{s5:dual_gap}
D_{\hx,\hy}(y):= f^*(y)+h(y) -f^*(\hy)-h(\hy) - \lr{K\hx,y - \hy} \geq 0
\quad \forall y\in Y,
\end{equation}
which is true by definition of $(\hx,\hy)$. Now we can rewrite
\eqref{s5:adding3} as
\begin{multline}\label{s5:before_cosine}
    \lr{x^{k+1} - x^k, \hx-x^{k+1}} +  \frac 1 \b \lr{y^{k+1} - y^k
    , \hy-y^{k+1}} +  \th_k \lr{x^{k} -
    x^{k-1}, x^{k+1}-\x^{k}}  \\ +
    \t_k \lr{K^*y^{k+1}- K^* y^k, \x^k - x^{k+1}}
     +\frac{\d}{2\b} \n{y^{k+1}-y^k}^2 -\frac{\t_k^2}{2} \n{K^*y^{k+1} -K^*y^k}^2
     \\ \geq \t_k \bigl((1+\th_k)P(x^k)-\th_k P(x^{k-1})+D(y^{k+1})\bigr).
\end{multline}
Applying cosine rules for all inner products in the first line in \eqref{s5:before_cosine} and using that $\th_k(x^k-x^{k-1})=\x^k-x^k$,
we obtain
\begin{multline}\label{s5:after_cosine}
   \frac{1}{2}(\n{x^k-\hx}^2 - \n{x^{k+1}-\hx}^2 - \n{x^{k+1}-x^k}^2)
   + \\ \frac{1}{2\b} (\n{y^k-\hy}^2 -
    \n{y^{k+1}-\hy}^2 - \n{y^{k+1}-y^k}^2 )\\ + \frac 1 2 ( \n{x^{k+1}-x^k}^2 - \n{\x^k-x^k}^2
    -\n{x^{k+1}-\x^k}^2) \\ +     \t_k \n{K^*y^{k+1}- K^* y^k}\n{ x^{k+1}-\x^k} +\frac{\d}{2\b} \n{y^{k+1}-y^k}^2 -\frac{\t_k^2}{2} \n{K^*y^{k+1} -K^*y^k}^2\\ \geq \t_k \bigl((1+\th_k)P(x^k)-\th_k P(x^{k-1})+D(y^{k+1})\bigr).
\end{multline}
Finally, applying Cauchy's inequality in \eqref{s5:after_cosine} and
using that $\t_k\th_k\leq \t_{k-1}(1+\th_{k-1})$, we get
\begin{multline}\label{s5:no_D2}
   \frac 1 2 \n{x^{k+1}-\hx}^2  + \frac{1}{2\b} \n{y^{k+1}-\hy}^2 + \t_k(1+\th_k)P(x^k)
   \leq \\  \frac 1 2 \n{x^k-\hx}^2 + \frac{1}{2\b} \n{y^k-\hy}^2+ \t_{k-1}(1+\th_{k-1}) P(x^{k-1}) \\- \frac 1 2 \n{\x^k-x^k}^2  -
    \frac{1-\d}{2\b} \n{y^{k+1}-y^k}^2,
\end{multline}
from which the convergence of $(x^k)$ and $(y^k)$ to a saddle point of
\eqref{saddle_general} can be derived in a similar way as in \Cref{th:main}.

\section{Numerical Experiments}\label{sec:experiment}
This section collects several numerical tests that will illustrate the performance of the proposed methods.
Computations\footnote{Codes can be found on \url{https://github.com/ymalitsky/primal-dual-linesearch}.} were
performed using Python 3 on an Intel Core i3-2350M CPU 2.30GHz running
64-bit Debian Linux 8.7.

For PDAL and APDAL  we initialize the input data as  $\mu = 0.7$, $\d = 0.99$, $\t_0 =
\frac{\sqrt{\min{\{m,n\}}}}{\n{A}_F}$. The latter is easy to compute
and it is an upper bound of $\frac{1}{\n{A}}$. The parameter $\b$ for
PDAL is always taken as $\frac{\s}{\t}$ in PDA with fixed steps $\s$
and $\t$. A trial step $\t_k$ in Step 2 is always chosen as $\t_k = \t_{k-1}\sqrt{1+\th_{k-1}}$.
\subsection{Matrix game}
We are interested in the following min-max matrix game:
\begin{equation}
    \label{eq:minmax}
    \min_{x \in \D_n}\max_{y\in \D_m} \lr{Ax, y},
\end{equation}
where $x\in \R^n$, $y\in \R^m$, $A\in \R^{m\times n}$, and $\Delta_m$,
$\D_n$ denote the standard unit simplices in $\R^m$ and $\R^n$,
respectively.

For this problem we study the performance of PDA, PDAL (\Cref{alg:A1}),  Tseng's FBF
method\cite{tseng00},  and PEGM\cite{malitsky2016proximal}. For  comparison we use the
primal-dual gap $\mathcal G(x,y)$, which can be easily computed for a
feasible pair $(x,y)$ as
$$\mathcal G(x,y) = \max_i (Ax)_i - \min_j (A^*y)_j.$$
Since iterates obtained by Tseng's method may be infeasible,
we used an auxiliary point (see \cite{tseng00}) to compute the
primal-dual gap.

The initial point in all cases was chosen as
$x^0 = \frac{1}{n}(1,\dots, 1)$ and $y^0 = \frac 1 m (1,\dots, 1)$. In
order to compute projection onto the unit simplex we used the
algorithm from~\cite{duchi2008efficient}. For PDA we use
$\tau = \sigma = 1/\n{A} = 1/\sqrt{\la_{\max}(A^*A)}$, which we compute
in advance. The input data for FBF and PEGM are  the same as in
\cite{malitsky2016proximal}.  Note that these methods also use a
linesearch.

We consider four differently generated samples of the matrix $A\in \R^{m \times n}$:
\begin{enumerate}
    \item  $m=n=100$. All entries of $A$
    are generated independently from the uniform distribution in $[-1,1]$.
    \item  $m=n=100$. All entries of $A$
    are generated independently from the the normal distribution $\mathcal{N}(0,1)$.

    \item $m=500$, $n=100$. All entries of $A$
    are generated independently from the normal distribution
    $\mathcal{N}(0,1)$.

    \item $m=1000$, $n=2000$. The matrix $A$ is sparse with $10\%$
    nonzero elements generated independently from the uniform
    distribution in $[0,1]$.
\end{enumerate}
For every case we report the primal-dual gap $\mathcal{G}(x^k,y^k)$ computed in every
iteration vs CPU time. The results are presented in \Cref{fig:matrix_game}.

\begin{figure}[ht]
    \begin{subfigure}[b]{0.49\linewidth}
        \centering
        \includegraphics[width=\linewidth]{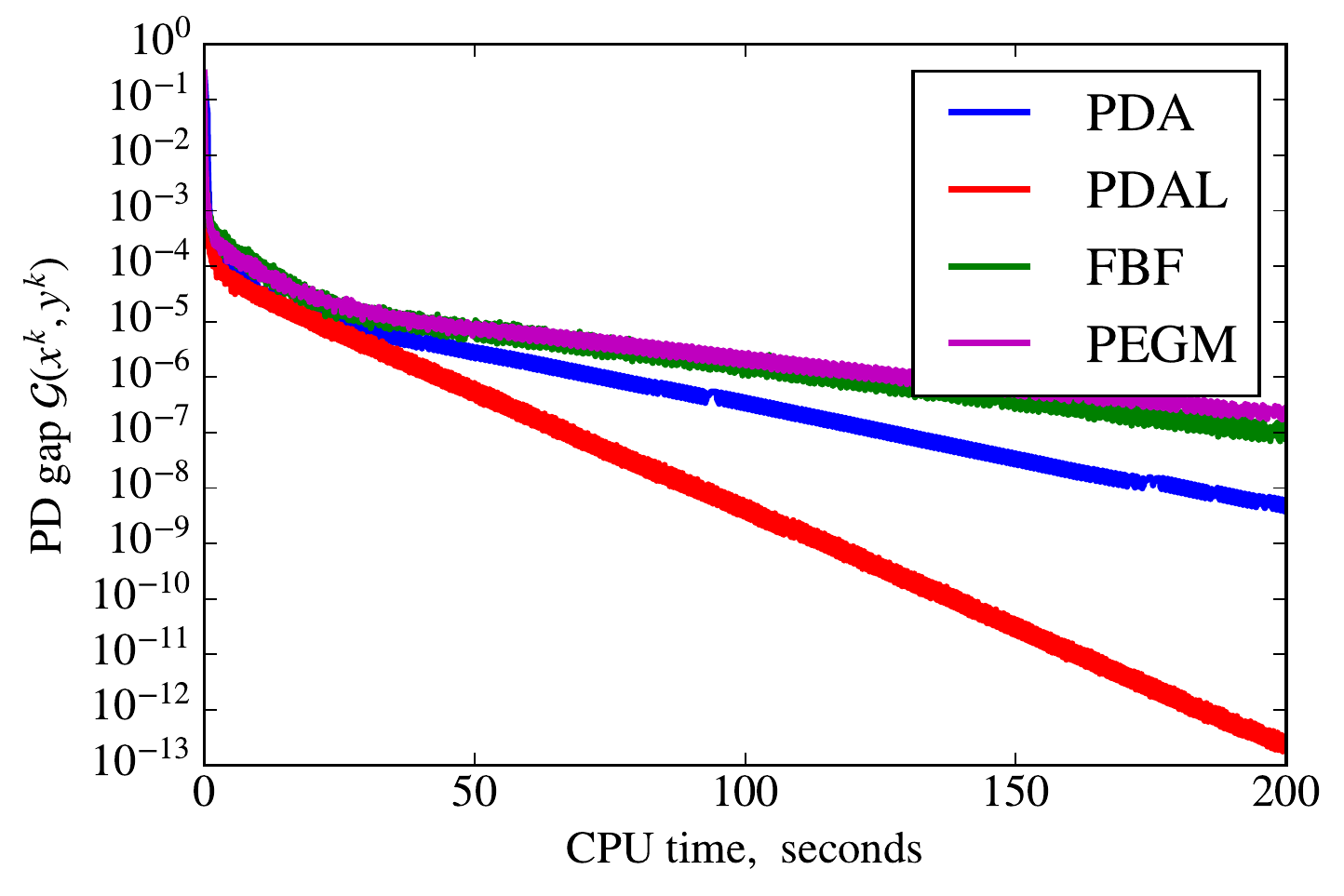}
        \caption{Example 1}
            \vspace{4ex}
    \end{subfigure}
    \begin{subfigure}[b]{0.49\textwidth}
            \centering
        \includegraphics[width=\linewidth]{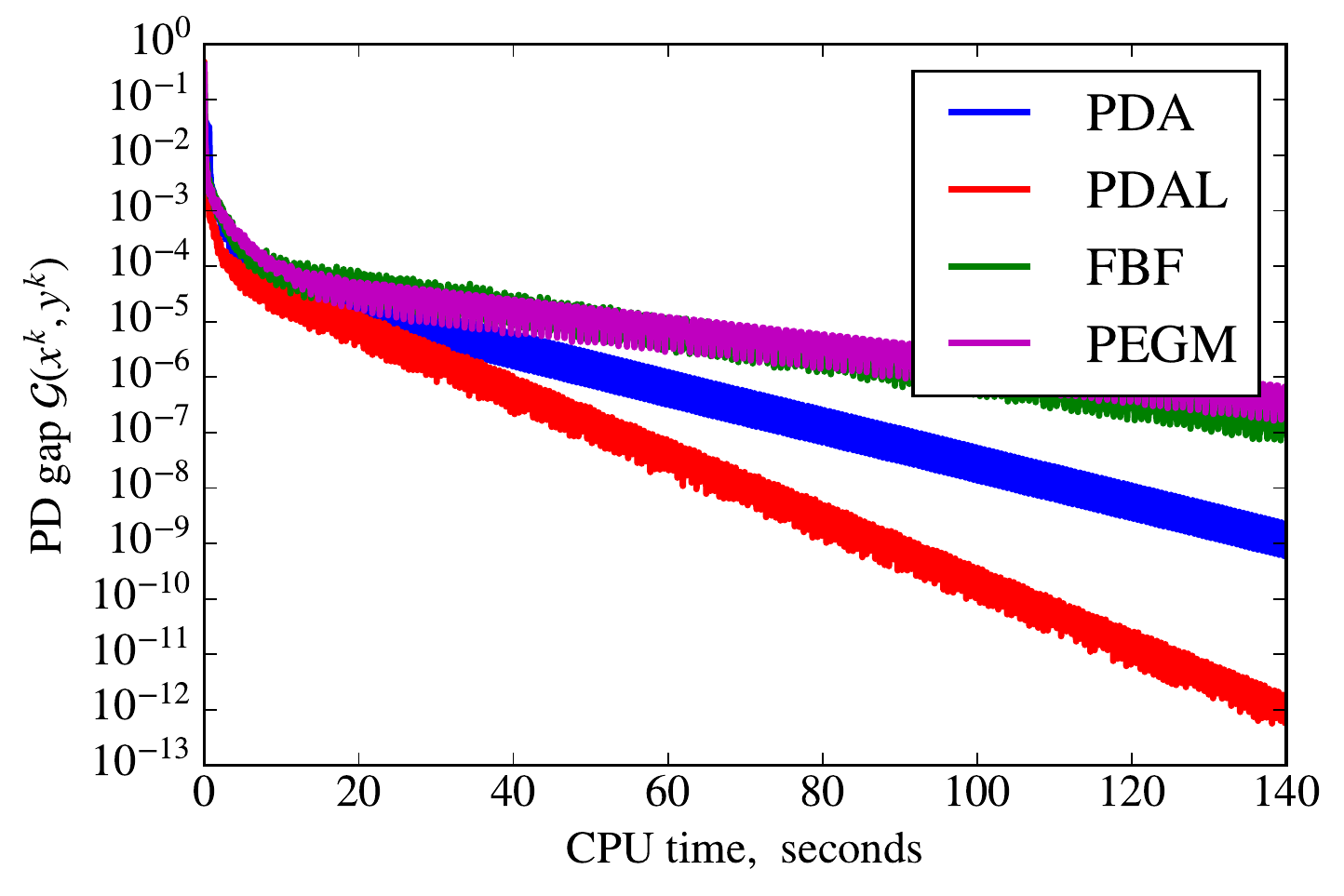}
        \caption{Example 2}
            \vspace{4ex}
    \end{subfigure}

    \begin{subfigure}[b]{0.49\linewidth}
            \centering
        \includegraphics[width=\linewidth]{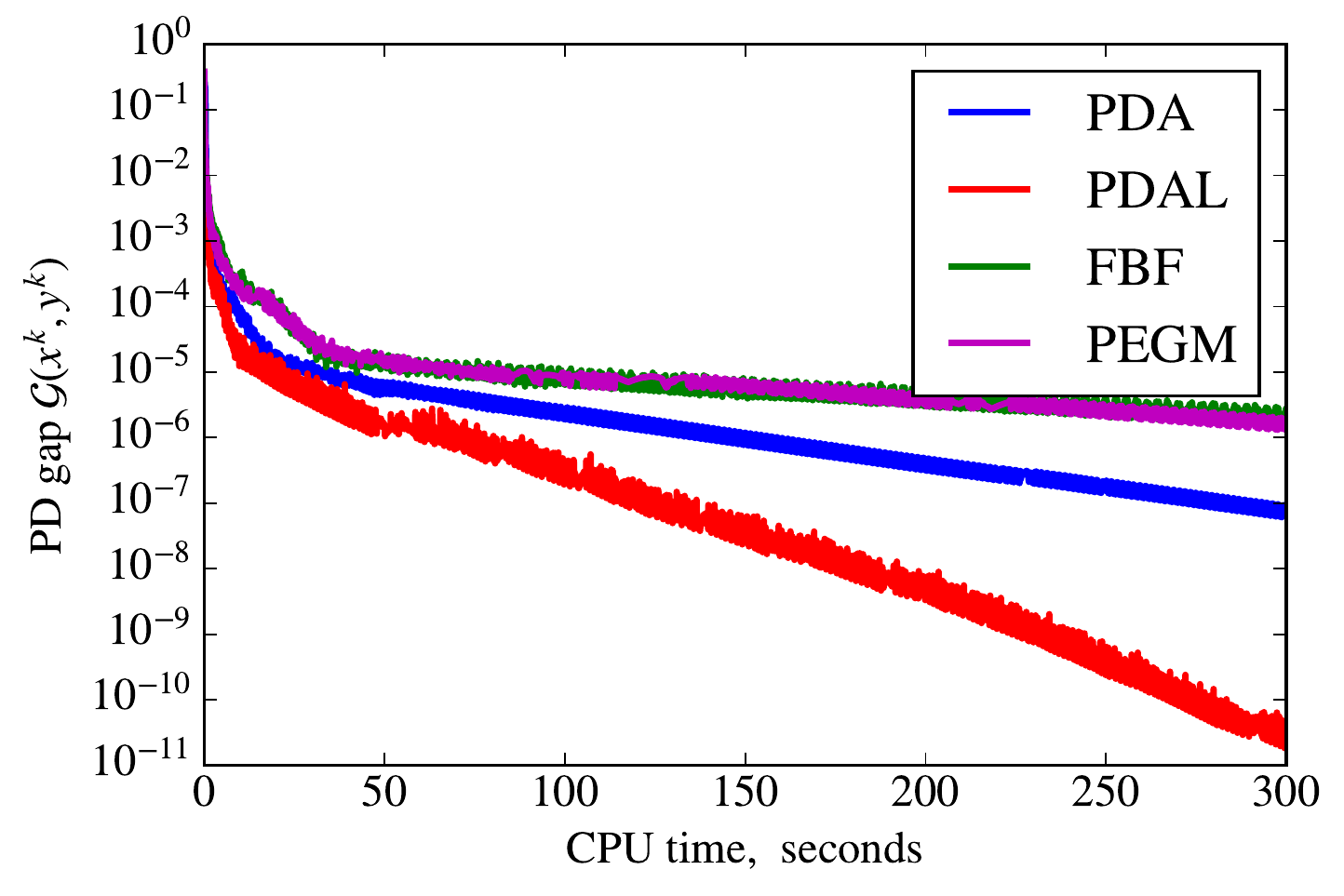}
        \caption{Example 3}
    \end{subfigure}
    \begin{subfigure}[b]{0.49\linewidth}
        \centering
        \includegraphics[width=\linewidth]{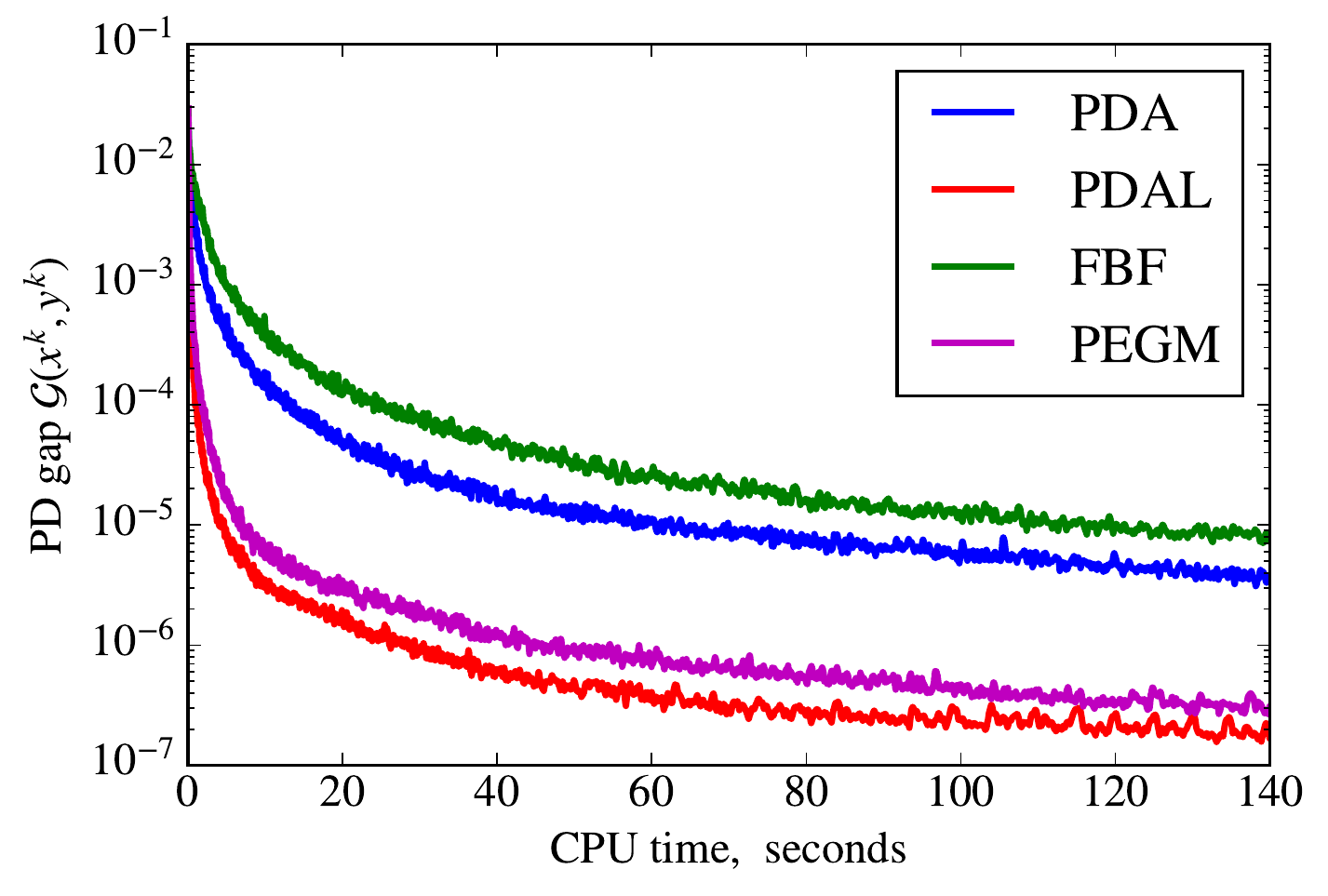}
        \caption{Example 4}
    \end{subfigure}
    \caption{Convergence plots for problem \eqref{eq:minmax}}
    \label{fig:matrix_game}
\end{figure}

\subsection{$l_1$-regularized least squares}
We study the following $l_1$-regularized problem:
\begin{equation}\label{l1-reg}
    \min_{x} \phi(x): = \frac 1 2 \n{Ax-b}^2 + \la \n{x}_1,
\end{equation}
where $A\in \R^{m\times n}$, $x\in \R^n$, $b\in \R^m$.
Let $g(x)=\la \n{x}_1$, $f(p)=\frac 1 2 \n{p-b}^2$. Analogously to
\eqref{least_square} and \eqref{least_square_saddle}, we can rewrite
\eqref{l1-reg} as
$$\min_x \max_y g(x)+\lr{Ax,y}-f^*(y),$$
where $f^*(y) = \frac 1 2 \n{y}^2 + (b,y) = \frac 1 2 \n{y+b}^2
-\frac{1}{2}\n{b}^2$. Clearly, the last term does not have any impact
on the prox-term and we can
conclude that $\prox_{\la f^*}(y) = \frac{y-\la b}{1+\la}$. This means
that the linesearch in \Cref{alg:A1} does not require any
additional matrix-vector multiplication (see \Cref{rem:linear}).

We generate four instances of problem \eqref{l1-reg}, on which  we
compare the performance of PDA, PDAL, APDA (accelerated primal-dual algorithm), APDAL, FISTA~\cite{fista}, and SpaRSA
\cite{wright2009sparse}. The latter method is a variant of the proximal
gradient method with an adaptive linesearch. All methods except PDAL
and SpaRSA require predefined step sizes.
For this we compute in advance $\n{A}=\sqrt{\la_{\max}(A^*A)}$.  For
all instances below we use the following parameters:
\begin{itemize}
    \item \emph{PDA}: $\sigma = \frac{1}{20\n{A}},
    \tau=\frac{20}{\n{A}}$;

    \item \emph{PDAL}: $\b =1/400$;

    \item \emph{APDA, APDAL}: $\b = 1$, $\c = 0.1$;

    \item \emph{FISTA}: $\a = \frac{1}{\n{A}^2}$;

    \item \emph{SpaRSA}: In the first iteration we run a standard
    Armijo linesearch procedure to define $\a_0$ and then we run SpaRSA
    with parameters as described in \cite{wright2009sparse}: $M=5$,
    $\s = 0.01$, $\a_{\max} = 1/\a_{\min} = 10^{30}$.
\end{itemize}

For all cases below we generate some random $w \in \R^n$ in which $s$
random coordinates are drawn from from the uniform distribution in
$[-10,10]$ and the rest are zeros.  Then we generate $\nu\in \R^m$
with entries drawn from $\mathcal{N}(0,0.1)$ and set $b = Aw+\nu$.
The parameter $\la = 0.1$ for all examples and the initial points  are $x^0 = (0, \dots , 0)$, $y^0=Ax^0-b$.

The matrix $A\in \R^{m\times n}$ is constructed in one of the following ways:
\begin{enumerate}
    \item $n=1000$, $m=200$, $s=10$. All entries of $A$ are generated
    independently from $\mathcal{N}(0,1)$.

    \item $n=2000$, $m=1000$, $s=100$. All entries of $A$
    are generated independently from $\mathcal{N}(0,1)$.

    \item $n=5000$, $m=1000$, $s=50$.  First, we generate
    the matrix $B$ with entries from $\mathcal{N}(0,1)$. Then for any
    $p\in (0,1)$ we construct the matrix $A$ by columns $A_j$,
    $j=1,\dots, n$ as follows: $A_1 = \frac{B_1}{\sqrt{1-p^2}}$,
    $A_j = p*A_{j-1}+B_j$. As $p$ increases,  $A$ becomes more
    ill-conditioned (see \cite{agarwal2010fast} where this example was
    considered). In this experiment we take $p=0.5$.
        \item The same as the previous example, but with $p = 0.9$.
    \end{enumerate}
    \Cref{fig:l1} collects the convergence results with
    $\phi(x^k)-\phi_*$ vs CPU time. Since in fact the value $\phi_*$ is
    unknown, we instead run our algorithms for sufficiently many iterations to
    obtain the ground truth solution $x^*$. Then we simply set $\phi_*
    = \phi(x^*)$.

\begin{figure}[ht]
    \begin{subfigure}[b]{0.49\linewidth}
        \centering
        \includegraphics[width=\linewidth]{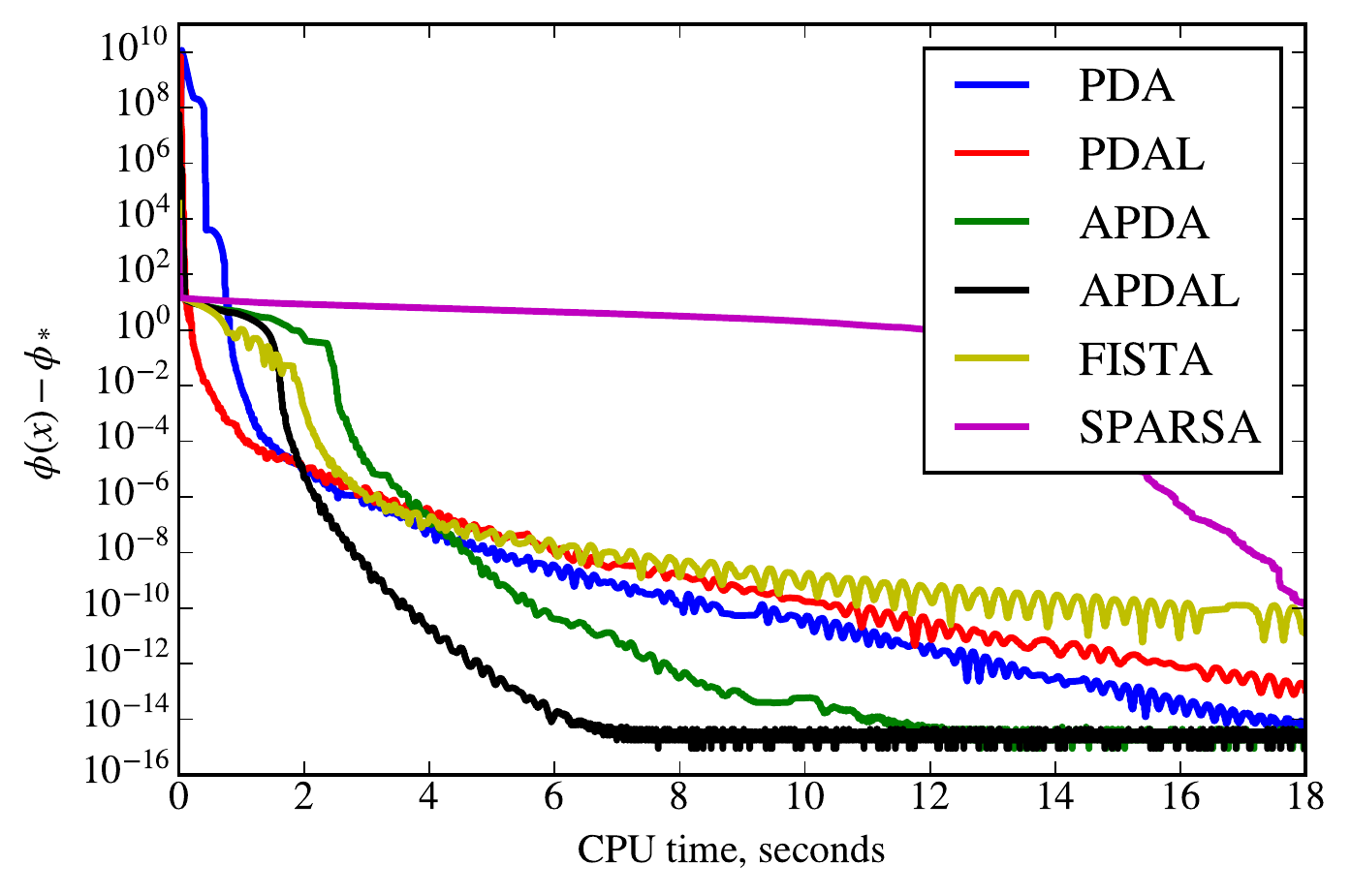}
        \caption{Example 1}
            \vspace{4ex}
    \end{subfigure}
    \begin{subfigure}[b]{0.49\textwidth}
            \centering
        \includegraphics[width=\linewidth]{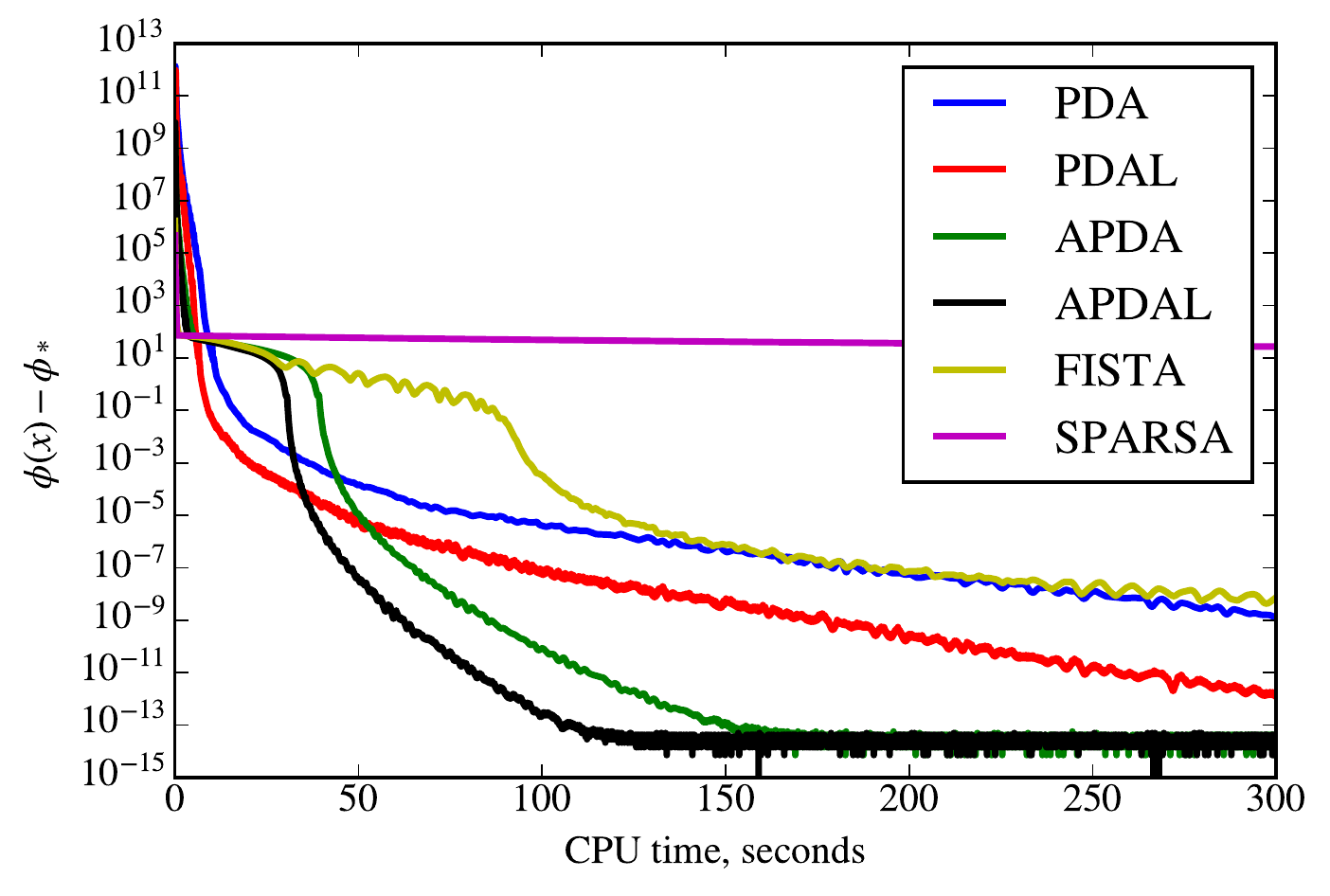}
        \caption{Example 2}
            \vspace{4ex}
    \end{subfigure}

    \begin{subfigure}[b]{0.49\linewidth}
            \centering
        \includegraphics[width=\linewidth]{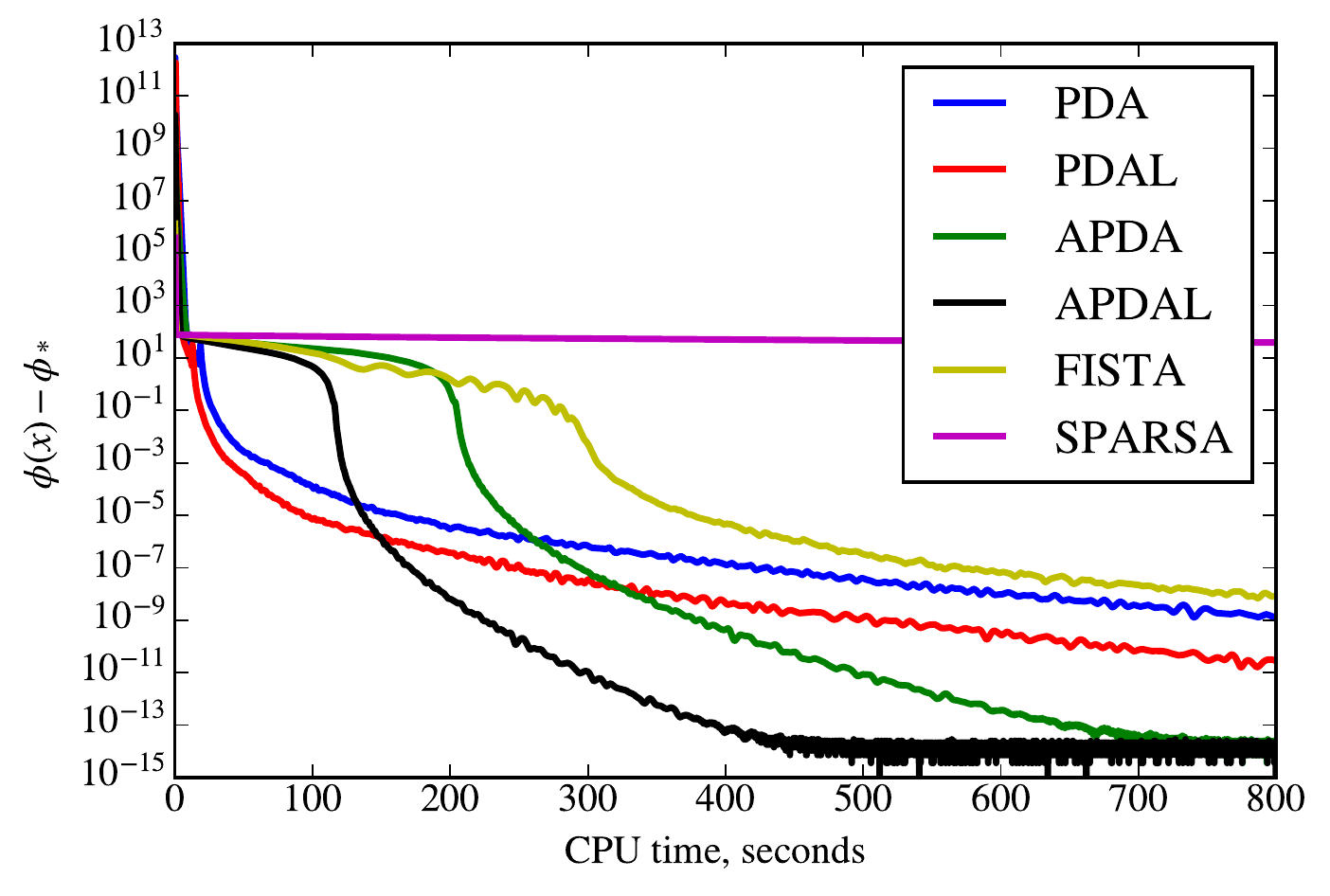}
        \caption{Example 3}
    \end{subfigure}
    \begin{subfigure}[b]{0.49\linewidth}
        \centering
        \includegraphics[width=\linewidth]{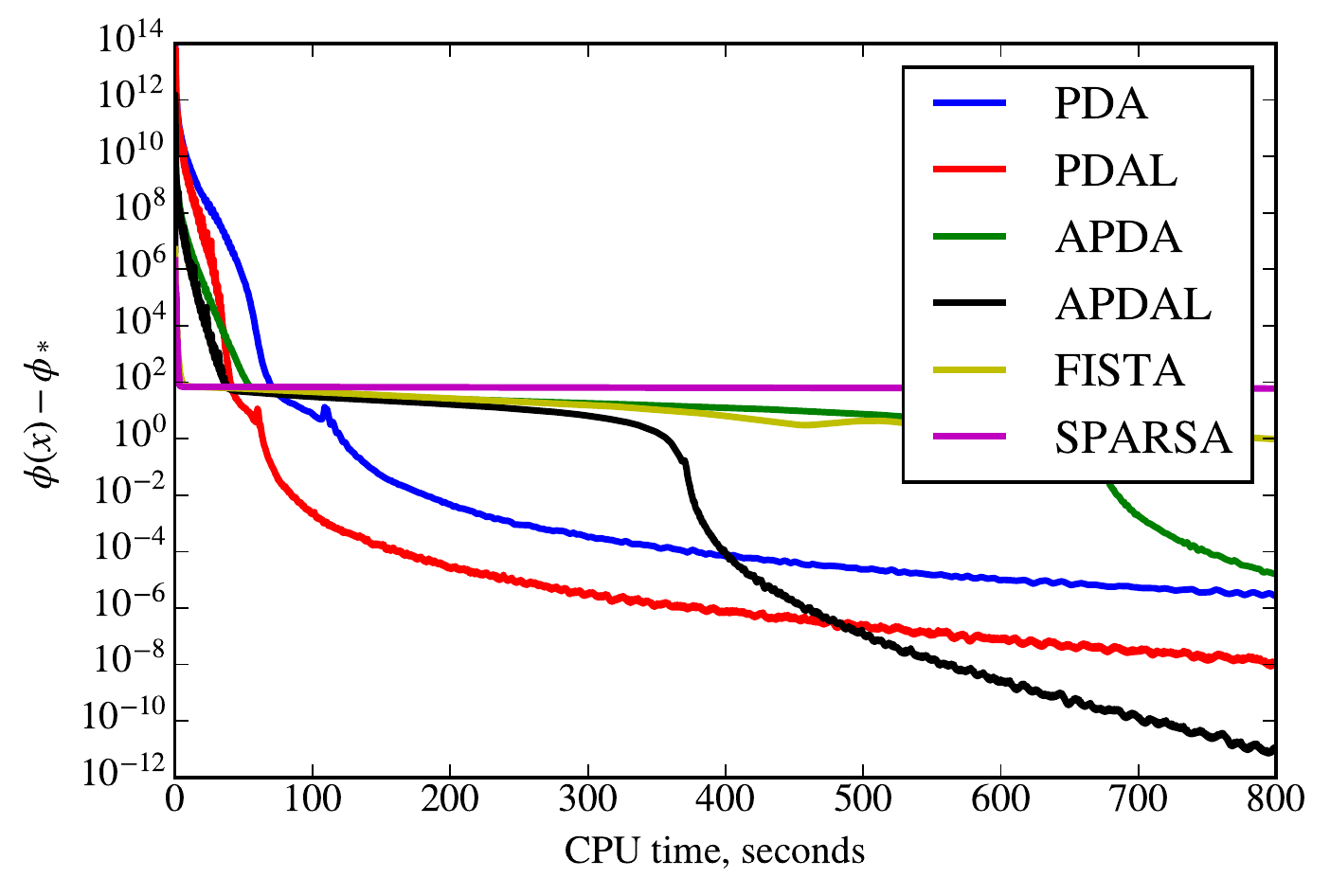}
        \caption{Example 4}
    \end{subfigure}
    \caption{Convergence plots for problem \eqref{l1-reg}}
    \label{fig:l1}
\end{figure}

\subsection{Nonnegative least squares}
Next, we consider another regularized least squares problem:
\begin{equation}\label{eq:nonneg}
  \min_{x\geq 0} \phi(x):=\frac 1 2 \n{Ax-b}^2,
\end{equation}
where $A\in \R^{m\times n}$, $x\in \R^n$, $b\in \R^m$.
Similarly as before, we can express it  as
\begin{equation}\label{noneg_square_saddle}
\min_x \max_y g(x)+\lr{Ax,y}-f^*y),
\end{equation}
where $g(x)=\d_{\R^n_+}(x)$, $f^*(y) = \frac 1 2 \n{y+b}^2$.
For all cases below we generate a random matrix $A\in \R^{m\times n}$
with density $d \in (0,1)$. In order to make the optimal value
$\phi_*=0$, we always generate $w$ as a sparse vector in $\R^n$ whose
$s$ nonzero entries are drawn from the uniform distribution in $[0,100]$.
Then we set $b=Aw$.

We test the performance of the same algorithms as in the previous
example. For FISTA and SpaRSA we use the same parameters as before.
For every instance of the problem, PDA and PDAL use the same
$\beta$. For APDA and APDAL we always set $\b = 1$ and $\c = 0.1$.
The initial points are $x^0 =(0,\dots,0)$ and $y^0=Ax^0-b=-b$.

We generate our data as follows:
\begin{enumerate}

    \item  $m = 2000$, $n=4000$, $d=1$, $s=1000$; the entries of $A$ are
    generated  independently from the uniform distribution in
    $[-1,1]$. $\b = 25$.

    \item $m=1000$, $n=2000$, $d=0.5$, $s=100$; the nonzero entries of $A$ are
    generated  independently from the uniform distribution in
    $[0,1]$. $\b = 25$.

    \item $m=3000$, $n=5000$, $d=0.1$, $s=100$; the nonzero entries of $A$ are
    generated  independently from the uniform distribution in
    $[0,1]$. $\b = 25$.

    \item $m=10000$, $n=20000$, $d=0.01$, $s=500$; the nonzero entries of $A$ are
    generated  independently from the normal distribution $\mathcal{N}(0,1)$. $\b =1$.
\end{enumerate}

As \eqref{eq:nonneg} is again just a regularized least squares problem,
the linesearch does not require any additional matrix-vector
multiplications. The results with $\phi(x^k)-\phi_*$ vs. CPU time are presented in \Cref{fig:nonneg}.
\begin{figure}[ht]
  \begin{subfigure}[b]{0.49\linewidth}
    \centering
    \includegraphics[width=\linewidth]{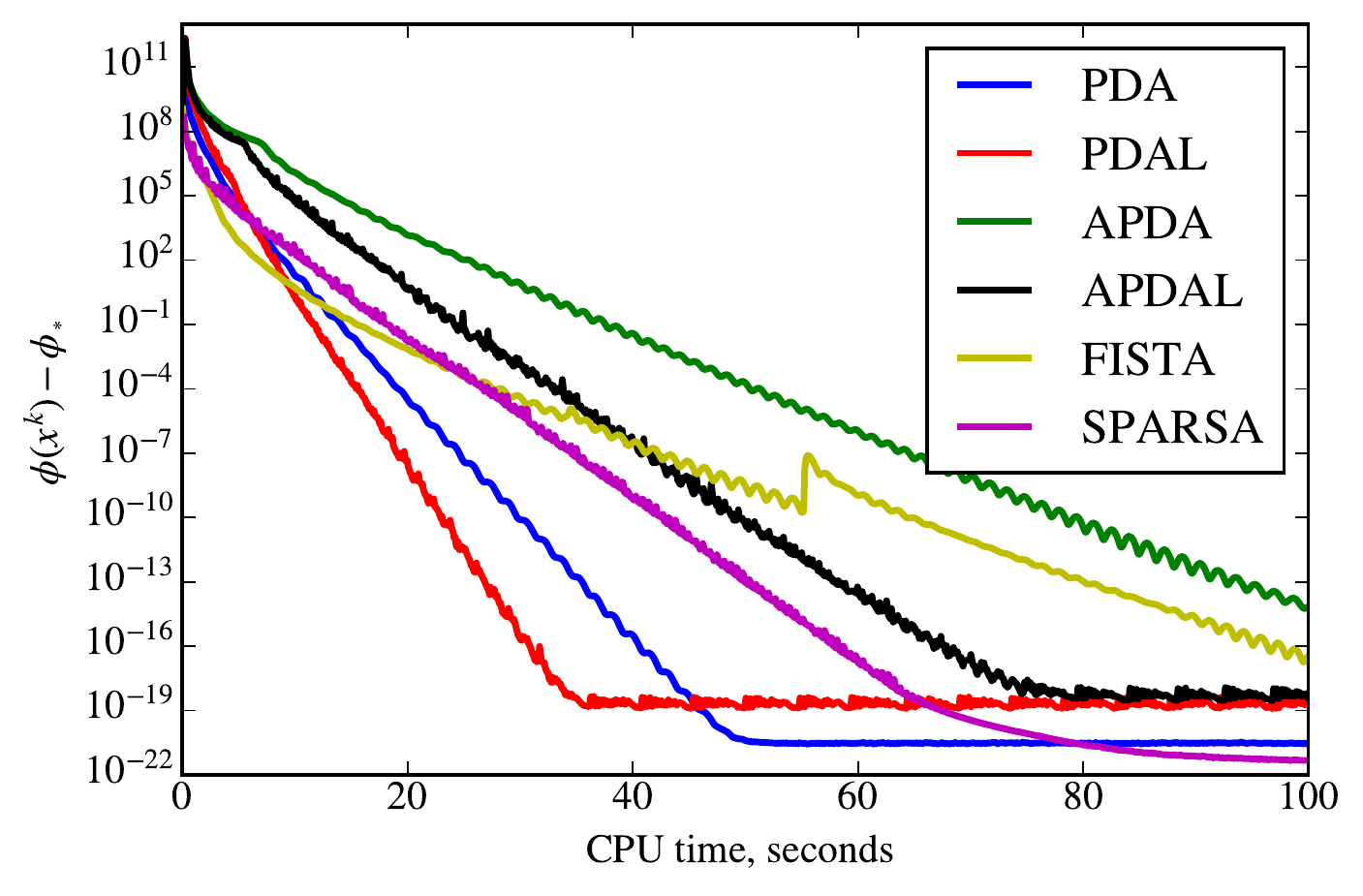}
    \caption{Example 1}
    \vspace{4ex}
  \end{subfigure}
  \begin{subfigure}[b]{0.49\linewidth}
    \centering
    \includegraphics[width=\linewidth]{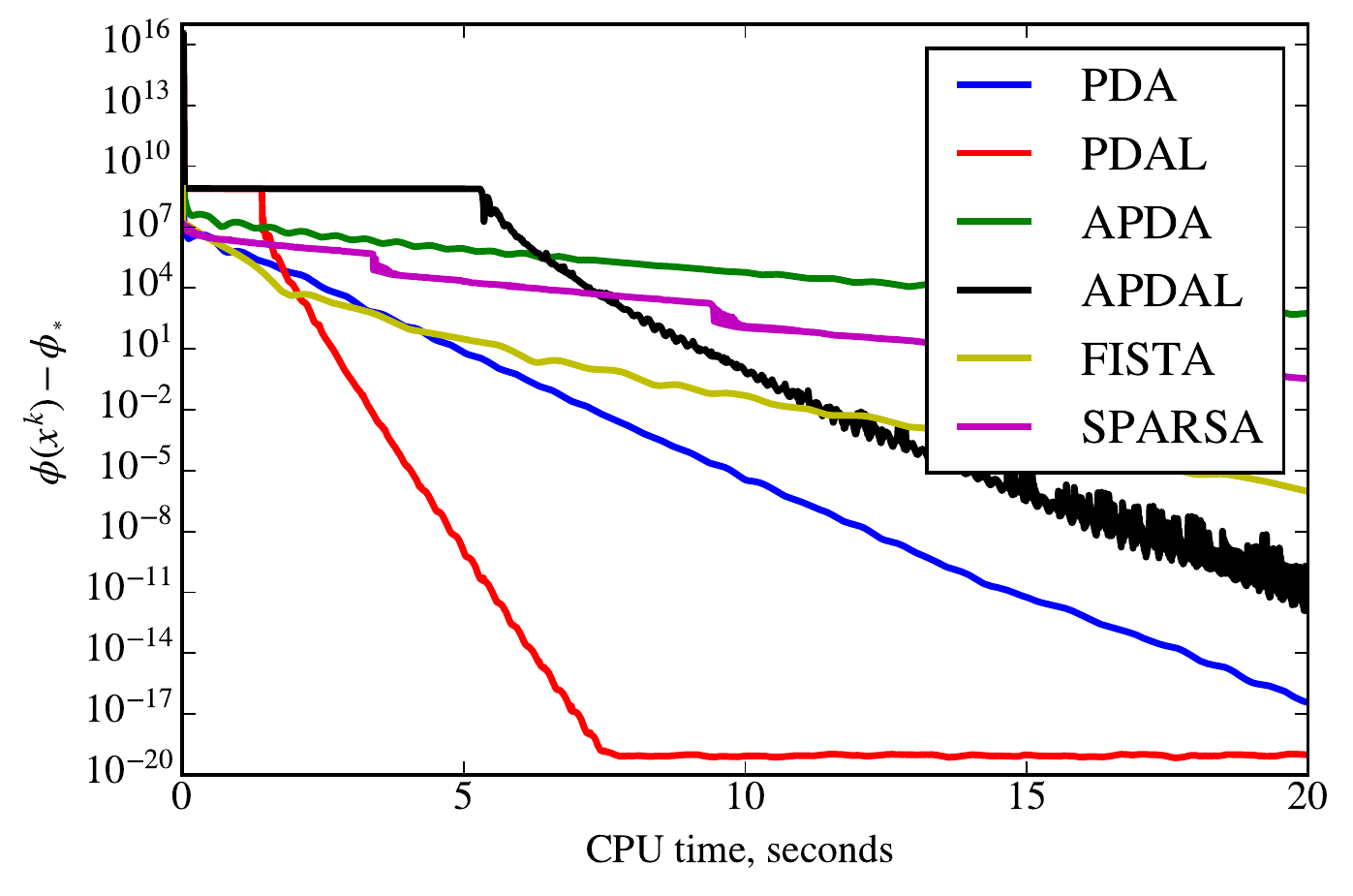}
    \caption{Example 2}
    \vspace{4ex}
\end{subfigure}

\begin{subfigure}[b]{0.49\linewidth}
    \centering
    \includegraphics[width=\linewidth]{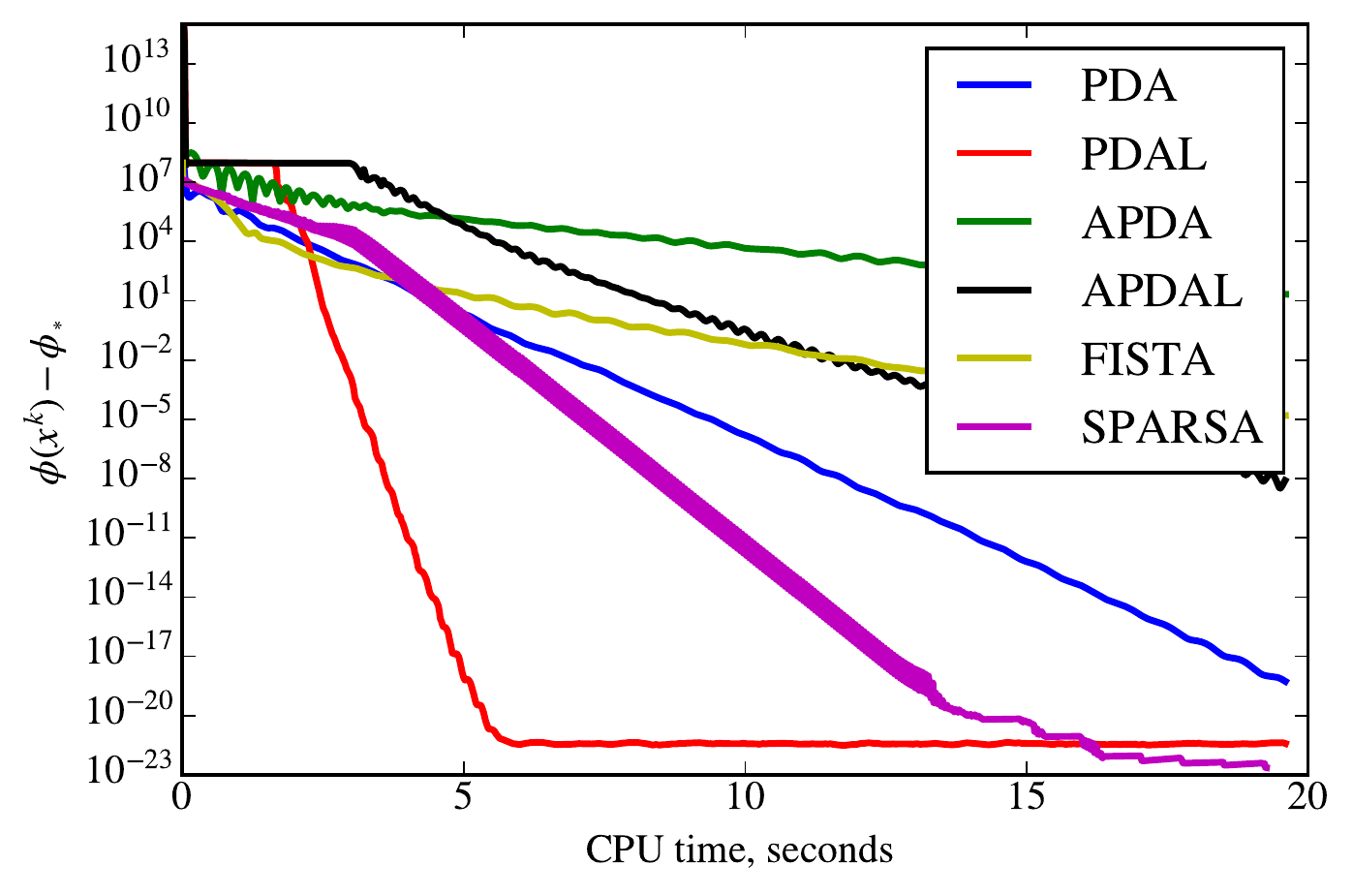}
    \caption{Example 3}
  \end{subfigure}
  \begin{subfigure}[b]{0.49\linewidth}
    \centering
    \includegraphics[width=\linewidth]{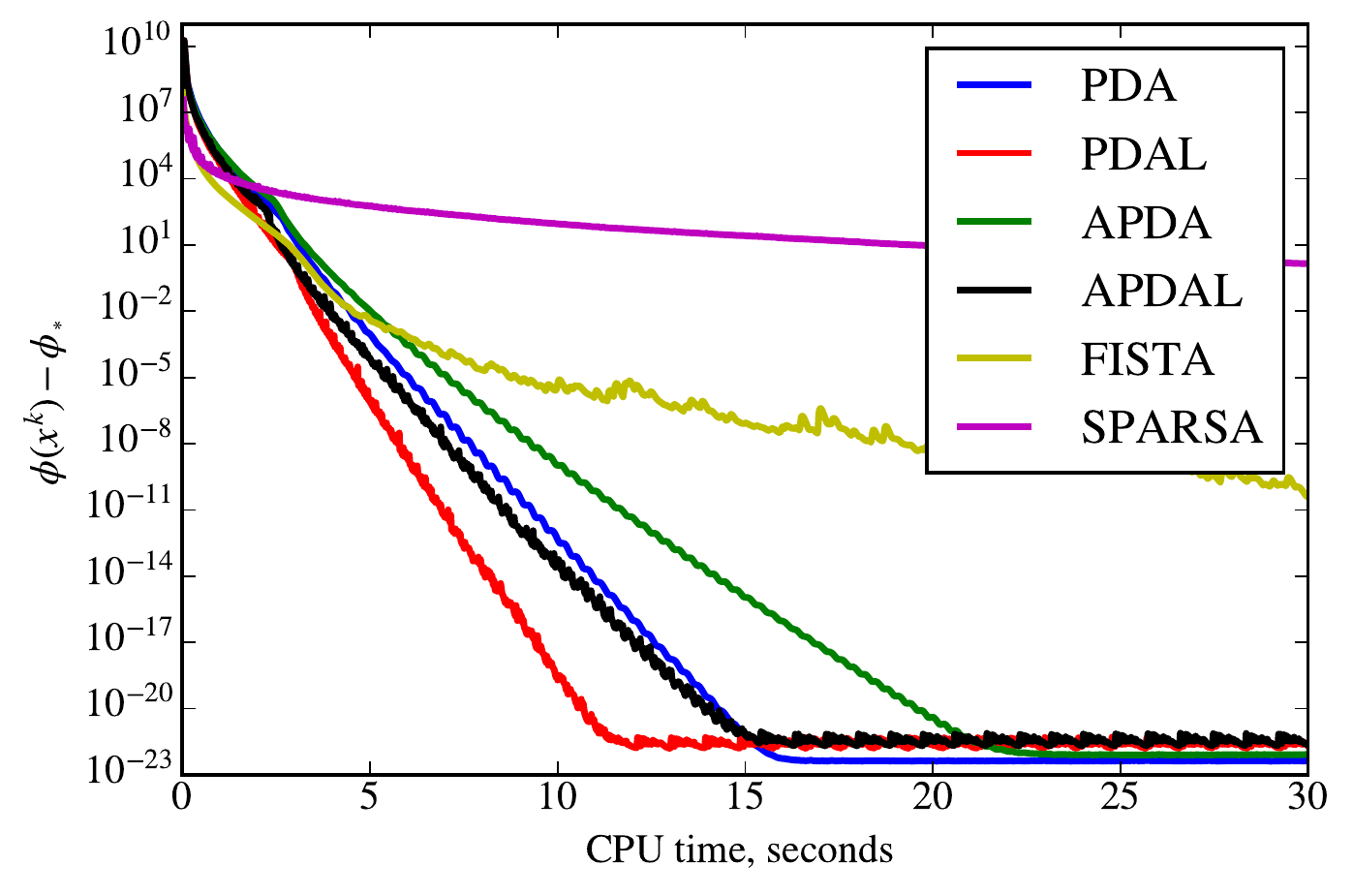}
    \caption{Example 4}
\end{subfigure}

  \caption{Convergence plots for problem \eqref{eq:nonneg}}
  \label{fig:nonneg}
\end{figure}

Although primal-dual methods may converge faster, they require tuning
$\beta = \s/\t$. It is also interesting to highlight that sometimes non-accelerated methods
with a properly chosen ratio $\s/\t$ can be faster than their
accelerated variants.

\section{Conclusion}
In this work, we have presented several primal-dual algorithms with
linesearch. On the one hand, this allows us to avoid the evaluation of the
operator norm, and on the other hand, it allows us to make
larger steps.  The proposed linesearch is very simple and in many important
cases it does not require any additional expensive operations (such as
matrix-vector multiplications or prox-operators).  For all
methods we have proved convergence. Our experiments confirm
the numerical efficiency of the proposed methods.

\vspace{0.5cm} {\noindent {\bfseries Acknowledgements:} The work is
supported by the Austrian science fund (FWF) under the project
"Efficient Algorithms for Nonsmooth Optimization in Imaging" (EANOI)
No. I1148. The authors also would like to thank the referees and the
SIOPT editor Wotao Yin for their careful reading of the manuscript
and their numerous helpful comments.  }

\bibliographystyle{abbrv}
\bibliography{publication}

\end{document}